\documentclass[12pt]{article}%
\usepackage{amsmath}
\usepackage{amsfonts}
\usepackage{amssymb}
\usepackage{graphicx}%
\providecommand{\U}[1]{\protect\rule{.1in}{.1in}}
\providecommand{\U}[1]{\protect\rule{.1in}{.1in}}
\providecommand{\U}[1]{\protect\rule{.1in}{.1in}}
\newtheorem{theorem}{Theorem}

\newtheorem{definition}[theorem]{Definition}

\newtheorem{lemma}[theorem]{Lemma}

\newtheorem{proposition}[theorem]{Proposition}

\newenvironment{proof}[1][Proof]{\noindent\textbf{#1.} }{\ \rule{0.5em}{0.5em}}
\begin{document}

\title{\textbf{Equilibrium policies when preferences are time inconsistent }}
\author{\textbf{Ivar Ekeland and Ali Lazrak\thanks{Ivar Ekeland is Canada Research Chair in
Mathematical Economics at the University of British Columbia
(ekeland@math.ubc.ca), Ali Lazrak is with the Sauder School of
Business at the University of British Columbia
(ali.lazrak@sauder.ubc.ca). We thank Graig Evans, Larry Karp, Jacob
Sagi and seminar participants at UBC, THEMA, Institut Henri
Poincar\'e, Society of Economic Dynamic (Vancouver 2006), CMS
(Winnipeg), The UBC annual finance conference (Whistler), the workshops ``optimization problems in financial
economics" (Banff) and, ``Risk: individual and collective decision
making" (Paris).}}}
\date{First version: June 2005\\ This version: August 2008}
\maketitle

\begin{abstract}
This paper characterizes differentiable and subgame Markov perfect
equilibria in a continuous time intertemporal decision problem with
non-constant discounting. Capturing the idea of non commitment by
letting the commitment period being infinitesimally small,  we
characterize the equilibrium strategies by a value function, which
must satisfy a certain equation. The equilibrium equation is
reminiscent of the classical Hamilton-Jacobi-Bellman equation of
optimal control, but with a non-local term leading to differences in
qualitative behavior. As an application, we formulate an overlapping
generations Ramsey model where the government maximizes a
utilitarian welfare function defined as the discounted sum of
successive generations' lifetime utilities. When the social discount
rate is different from the private discount rate, the optimal
command allocation is time inconsistent and we retain subgame
perfection as a principle of intergenerational equity. Existence of
multiple subgame perfect equilibria is established. The multiplicity
is due to the successive governments' inability to coordinate their
beliefs and we single out one of them as (locally)
renegotiation-proof. Decentralization can be achieved with both age
and time dependent lump sum transfers and, long term distorting
capital interest income taxes/subsidy.
\end{abstract}

\clearpage

\section{Introduction}

Time inconsistency is present in many dynamic decision making
problems\footnote{In particular the issue is of economic relevance
in the monetary and fiscal policy design of a benevolent government
(Kydland and Prescot \cite{Kyd}, Calvo \cite{Calvo}, Fischer
\cite{Fischer} and, Chari and Kehoe \cite{ChariKehoe}), the pricing
of a durable good for a monopolist (Stokey \cite{Sto}), the
ownership policy for a large shareholder (DeMarzo and
Uro\v{s}evi\'{c} \cite{DeMa}), the long term environmental decision
making (Chichilnisky \cite{Chichi} and, Li and L\"{o}fgren
\cite{LiLo}) and, the consumption saving private decision under
hyperbolic discount (Laibson \cite{Laibson1}) as well as the Ramsey
growth model when the social planner himself exhibits hyperbolic
discount (Krusell, Kuru\c{s}\c{c}u and Smith \cite{Krusell-Smith
0}).}. The primary objective of this paper is to develop a new
approach to analyze a class of time inconsistent games where time
inconsistency is due to non constant discount rates. Using this
methodology, we characterize the equilibria of the game and, for a
particular specification of the discount function, we establish
their existence and report a new type of
indeterminacy for this class of games. 
 The framework is the standard deterministic and stationary
Ramsey general equilibrium model of growth and capital accumulation
(see \cite{Ramsey}). Time inconsistency is due to the planner's
non-constant discount rates and there are several good reasons for
which this may occur. First, when the planner is an individual,
there is experimental evidence from psychology (e.g. Ainslie
\cite{Anslie1}) which challenges the assumption of constant discount
rates. In particular, there is robust evidence that people can
indulge in immediate gratification even if the delayed cost is high.
This suggest a revealed discount rate which is declining over time
which results in a \emph{hyperbolic} discount function. Second, if
the planner is a government, utilitarianism naturally leads to non
constant discounting and time inconsistency in models with multiple
generations. The source of time inconsistency is the inability of
the forward looking governments to account for the preferences of
the deceased cohorts.\footnote{A time inconsistency induced
by the structure of the preferences is also present for capital accumulation
models with non overlapping and altruistic generations when the
social planner uses the maxi-min Rawlasian criterion (Dasgupta
\cite{Dasgupta}). A similar situation also arises for consumption saving problems
when the preferences over consumption streams are non
additively separable (Kihlstrom \cite{ Kihlstrom}).} The idea is
intuitively discussed in the context of natural ressources
management by Sumaila and Walters \cite{Sumaila} and, explained more
formally in the optimal growth context in a model with altruistic
and non-overlapping generations (Bernheim \cite{Bernheim}) and, in a
model with non altruistic and overlapping generations (Calvo and
Obstfeld \cite{CalvoObstfeld}).

Following the previous literature on the topic, Strotz
\cite{Strotz}, Pollak \cite{Pollak}, Phelps and Pollak
\cite{Phelps-Pollak} and, Peleg and Yaari \cite{Peleg-Yaari}, we
search for subgame perfect equilibria of a dynamic game where
decision makers cannot bind their choices (non commitment) and are
aware of their inconsistency problem\footnote{See Gul and
Pesendorfer \cite{Gul1} for an alternative interpretation of the
Strotz's model.}. We restrict the planner to choose  pure strategies
that are continuously differentiable functions of contemporaneous
capital stock, making the framework as close as possible to the
standard optimal growth model. The notion of non-commitment is not
easily defined in continuous time: there is no notion of
``successor" because no matter how close in time two planners are,
there is always a third planner who precedes one of them and
succeeds the other. As a result, the discrete time induction
approach cannot be used to solve for the equilibrium and we have to
adapt it to define the concept of equilibrium in continuous time.
This is achieved with an intuitive construction: we assume that the
planner at any point in time can control his immediate successors,
thereby forming a coalition of a given small size that clearly
separates the current planner from the more distant ones. A strategy
is an equilibrium if it is the best strategy for the current planner
when the coalition is vanishingly small, so long as the same
strategy is expected to be used by the distant planners. The idea
conceptually parallels Aumann's continuum of consumers \cite{Aumann}
and it simply reduces to the Barro's \cite{Barro} approach when
linear equilibria exist.\footnote{See also Phelps \cite{Phelps} for
an early discussion on the issue.} In the first part of this paper
(Sections $2$, $3$ and $4$) we provide a precise sense in which the
game can be formulated and, a characterization of equilibrium
policies. The characterization consists of an instantaneous
saving-consumption indifference condition together with a non linear
differential equation displaying a non local term (that is a term
which depends on the global behavior of the solution). The non local
term reflects the strategic motive of investing and, when the
discount rate is constant it becomes a local term and the equation
reduces to the familiar Hamilton-Jacobi-Bellman equation (HJB) of
optimal control. The equation extends the Generalized Euler Equation
of Harris and Laibson \cite{Harris-Laibson} concerning the case where the
planner is facing a countable number of successors to a case where
the planner is facing a continuum of successors. Notice that an alternative
approach could have been adopted for our class of continuous time
dynamic games. It is possible to consider first a discrete time game
on a time grid, define the equilibrium by using induction methods on
the grid, and consider the equilibrium policy resulting from letting
the gird be vanishingly fine. This approach has been successfully
adopted by Luttmer and Mariotti \cite{Luttmer-Mariotti} for linear
equilibria of consumption-saving games under
uncertainty.\footnote{Simon and Stinchcombe \cite{Simon} also used
the infinitesimal grid approach to define a more general class of
games and history dependent strategies in continuous time.} Indeed
the linear equilibria of Luttmer and Mariotti
\cite{Luttmer-Mariotti} are mutually consistent with ours in a
deterministic version of their model.  We did not pursue this path
here because the approach that we undertook uses simple marginal
calculations involving standard differential calculus techniques and
does not require an extraneous convergence theorem. More
importantly, it directly leads us to an equilibrium characterization
that mirrors the HJB equation of optimal control which suggests that
our method is the ``adequate" notion of dynamic games in continuous
time.\footnote{Given the convergence result of Luttmer and Mariotti
\cite{Luttmer-Mariotti} for discrete time linear equilibria to our
solution, there is a suspicion that the convergence holds for non
linear equilibria as well.}

What appears to be new in this first part of the paper (Sections
$2$, $3$ and $4$) relative, for instance, to Barro \cite{Barro} and,
Luttmer and Mariotti \cite{Luttmer-Mariotti} is the formal model of
the game in continuous time and the resulting general
characterization of equilibrium policies. The framework allows for
non linear policies and the equilibrium characterization addresses
directly the Pollak's criticism of Strotz's work (see section $4$ of
Pollak \cite{Pollak}). More importantly,  we believe the general
methodology that we develop here can be applied in a broad set of
time inconsistent problems including the problem of optimal fiscal
policy. To illustrate the usefulness of the methodology, we develop
an application in the context of the Blanchard \cite{Blanchard}
model of economic growth with finite life (the perpetual youth
model). The Blanchard model is a continuous time version of the
overlapping-generations models of Samuelson \cite{Samuelson} and
Diamond \cite{Diamond} offering more tractability for the
aggregation of individual variables. We explore the dynamic
allocation problem of a utilitarian government (the planner)
maximizing the discounted sum of the surviving and unborn
generations' lifetime utilities. As we mentioned earlier,
implementation of utilitarian welfare optima in overlapping
generation models is problematic. Except in few special cases,
utilitarianism renders the planner's objective time inconsistent.
Time consistency is achieved for example when the planner's discount
rate is equal to the individual discount rate or if the planner's
discount function satisfies some unnatural assumptions (Calvo and
Obstfeld \cite{CalvoObstfeld}). The purpose of our application is to
analyze the centrally planned allocation when the social discount
function is aligned with the way in which economists usually model
utilitarianism that is, an exponential forward looking discount
function with a constant rate of growth. The optimal command is then
time inconsistent and the planner is forced to approach the
allocation decision as a strategic problem. Our principle of
intergenerational equity is then anchored in the concept of subgame
perfection of  the extensive form game of successive planners. This
approach has already been taken for Rawlasian welfare (see Dasgupta
\cite{Dasgupta}, Lane and Mitra \cite{LaneMitra} and, Asheim
\cite{Asheim2}). To our knowledge, our application is the first to
address this preference based time inconsistency friction in an
overlapping generations economy with a utilitarian government.

We now summarize the main finding from our application (section 5)
that we carry out in two steps.

First, we describe how the centrally planned economy operates.  The
government has to solve two questions: how to allocate consumption
across the surviving cohorts and how to allocate aggregate
consumption over time. The inspection of this dual decision shows
that when the allocation across surviving generations is restricted
to be stationary and linear in aggregate consumption, the problem of
allocating the aggregate consumption over time is in fact isomorphic to an
infinitively lived representative agent growth model with non
constant discount of the type analyzed in the first part of the
paper. The social discount function takes the form of a mixture of
two exponential functions. Due to this special structure we can
reduce the equilibrium characterization from the first part of the
paper to a system of two coupled ordinary differential equations
which do not involve non local terms. Taking advantage of this
special structure and using the central manifold theorem (Carr
\cite{Carr}), we then prove the existence of multiple equilibria.
The equilibrium policies are continuous and differentiable in the
capital stock and they inherit the smoothness (in a sense to be made
precise) from the underlying production function. Equilibrium
multiplicity results in a continuum of possible steady state level
of capital stock within the range of an open and bounded interval.
When the planner's discount rate is equal to the individual discount
rate, the interval shrinks to one point, and the capital stock
converges to its modified golden rule level. The driving factor for
multiplicity is the governments' inability to coordinate their
expectations on any policy. When adapted to our continuous time
framework, renegotiation-proofness is restrictive as in Kocherlakota
\cite{Kocherlakota}. Under this refinement, all successive
governments use in agreement, amongst all equilibrium policies, the
one that induces the capital stock to converge to the highest steady
state level. The renegotiation-proof steady state level of capital
stock depends on the government discount rate, the private discount
rate and the individuals' life expectancy. When the government
discount rate converges to $0$, the capital stock path resulting
from the renegotiation-proof equilibrium converges to its golden
rule level.

Second, we discuss how to find a tax schedule that, in a market with
actuarially fair annuities (Blanchard \cite{Blanchard}) places the
economy on the desired disaggregate path of accumulation. If the
social discount rate is equal to the private discount rate,
distortionary taxation is not required by the second welfare
theorem. If, as we suppose in our application the government is more
patient than the individuals,\footnote{See Caplin and Leahy
\cite{Caplin-Leahy1} for a non paternalistic argument supporting the
idea that social planner can be more patient than private
individuals.} the time inconsistency problem creates a wedge between
the cost and benefits of saving at the individual level. Under this
assumption, the laissez-faire economy cannot decentralize the
allocation and we show that such decentralization can be achieved if
the government uses a date and age conditioned lump sum taxation and
distorting capital income taxation. We provide a closed-form expression
for the required path of capital income tax rates that shows that
its long term level can be positive or negative. The result is
clearest in one specification of the model where the the consumption
allocation across the surviving cohorts is egalitarian. Under this
specification, as in the infinitely lived identical agents economy,
there is no heterogeneity in consumption at any point in time. Yet
the government finds it optimal to subsidize the capital income with
an age independent rate, including in the steady state. The motive
for the subsidy is to give the savings incentives to the private
agents who, from the point of view of the government, are not saving
sufficiently. The qualitative conclusion from our application is
that the preference based time inconsistency friction faced by a
utilitarian government creates in its own a role for capital
taxation in the long term. The results is in stark contrast with the
benchmark infinitively lived agents economy where capital income
taxation should be zero in the long term (Judd \cite{Judd} and
Chamley \cite{Chamley}).

{\bf Related Literature.} Before turning to the model, we summarize
how our work relates to the literature on which it builds. Our main
result is the existence of multiple equilibria in a dynamic game. We
introduce a definition of continuous time games and we prove
existence by using a new approach, based on the central manifold
theorem (Carr \cite{Carr}), in a deterministic Ramsey growth model.
Our paper naturally relates to the game theoretic literature on
Markov-consistent plan (MCP). An MCP is a subgame perfect
equilibrium of the extensive form game between the successive
planners in which the strategy is pure and only depends on the
payoff relevant variables (capital stock). the MCPs do not always
exist in time inconsistent decision making problems. Peleg and Yaari
\cite{Peleg-Yaari} constructed a simple finite horizon
intra-personal game where MCPs do not exist.\footnote{However, more
general history dependent equilibrium strategies still exist as
Goldman \cite{Goldman} showed in a finite horizon setting and Harris
\cite{Harris} showed in an infinite horizon setting} Linear MCPs
have however been reported in the hyperbolic discount
literature\footnote{This literature shows that apparent
irrationality of individuals, even in financial markets, can be
ascribed to the fact that the psychological discount factor is not
exponential; see Laibson \cite{Laibson2}, Harris and Laibson
\cite{Harris-Laibson1}, Diamond and Koszegi \cite{Diamond-Koszegi}
and others.} in a frictionless consumption saving problems with
homothetic time additive utilities and linear production functions.
In a deterministic setting, linear MCPs have been reported in
Laibson \cite{Laibson0} in a discrete time consumption saving
problem as well as in Barro \cite{Barro} in the context of an
infinite horizon continuous time decentralized version of the Ramsey
growth model. Luttmer and Mariotti \cite{Luttmer-Mariotti} also
reported linear MCPs in an infinite horizon endowment asset pricing
model with uncertainty. The aggregate temporal allocation problem of
our overlapping generations growth model may be interpreted as a
consumption saving problem with hyperbolic discount and a non linear
state dynamic for wealth. With this interpretation, our application
can be seen as a an extension of Laibson \cite{Laibson0}, Barro
\cite{Barro} and Luttmer and Mariotti \cite{Luttmer-Mariotti} to non
linear MCPs. Our results suggest then that the observational
equivalence results implied by the existence of linear MCPs, does
not hold when the technology is non linear. In the presence of non
linear technology, the MCP dynamic path emerging from our analysis
are not possible to reproduce with constant discount.

Non observational equivalence is also present in the infinite
horizon buffer stock models with income uncertainty, borrowing
constraints and hyperbolic discounting in discrete time (Harris and
Laibson \cite{Harris-Laibson}) and continuous time (Harris and
Laibson \cite{Harris-Laibson2}). These two papers show existence and
uniqueness of smooth MCPs when the discount function is sufficiently
close to the exponential discount.\footnote{When the discount
function is quasi-geometric (Phelps and Pollak \cite{Phelps-Pollak})
and sufficiently close to a geometric one, Harris and Laibson
\cite{Harris-Laibson} used tools from the bounded variations
calculus to Show existence and uniqueness of the equilibrium
strategies. Using alternative techniques,  Harris and Laibson
\cite{Harris-Laibson2} proved existence and uniqueness of smooth
equilibria for a stochastic discount function when it is in some
sense sufficiently close to the exponential discount function (this
is what they coin the instantaneous gratification model). It appears
that both proofs require the presence of income uncertainty.} Our
results do not overlap with theirs since we allow for non linear
technology and we do not have a borrowing constraint.
Methodologically, our proof of existence is different and does not
require the presence of uncertainty or the discount function to be
sufficiently close to the exponential function. Our results is also
different because multiplicity emerges as a central feature of the
equilibrium in our context.

Another type of MCPs existence results is found in the literature on
altruistic generations growth economies. In this context, existence
of MCP has been established for a finite horizon setting (Bernheim
and Rey \cite{Bernheim-Rey}) and for an infinite horizon setting
(Bernheim and Rey \cite{Bernheim-Rey1}) (see also Caplin and Leahy
\cite{Caplin-Leahy} for a recent related MCP's existence result
under uncertainty). In contrast to our approach, the existence
result of those papers hinges critically on introducing production
uncertainty.

Our paper also relates to the literature on the Ramsey growth model
when the central planner himself displays hyperbolic discount.
Krusell, Kuru\c{s}\c{c}u and Smith \cite{Krusell-Smith 0} undertook
an elegant comparison of a decentralized and a centralized version
of the Ramsey model with quasi-geometric discount, in discrete
time\footnote{See also a related paper by Judd \cite{Judd}.}. The
analysis in Krusell, Kuru\c{s}\c{c}u and Smith \cite{Krusell-Smith
0} does not cover the fundamental issue of multiple equilibria
because it is assumed that the equilibria must be the limit of
finite-horizon equilibria. We do not impose this restriction in our
setting and the infinite horizon naturally underscores multiple
equilibria. Multiplicity is also discussed in Karp \cite{Karp} and
Krusell and Smith \cite{Krusell-Smith}. In a continuous time model ,
Karp \cite{Karp} obtained the MCP's necessary conditions in a growth
model with non constant discount rate by considering first the
equilibrium of a sequence of planners in discrete time and then he
took the continuous time limit. While the passage to the limit is
not mathematically justified in Karp \cite{Karp},  his equilibrium necessary conditions
are consistent with Our incremental contribution
relative to his is that we took the necessary theoretical steps to
define the novel notion of continuous time game and, as a byproduct,
we proved existence of multiple equilibria. Krusell and Smith
\cite{Krusell-Smith} also report multiple equilibria in a Ramsey
growth model with quasi-geometric discount. However, their MCPs are
supported by discontinuous consumption policies whereas our MCPs are
continuously differentiable policies. Therefore, our analysis
suggests that the multiplicity is somewhat more fundamental because
it does not need to be structured around discontinuous saving rules.

Finally, our application also relates to the literature on optimal
fiscal policy. In the context of a growth model with
infinitely-lived individuals, Judd \cite{Judd} and Chamley
\cite{Chamley} established that capital taxation should be zero in
the long term. In contrast to these result,  we suggest that in
overlapping generations economies, the preference based time
inconsistency friction faced by the government creates a role for
long term capital income taxation.\footnote{The literature on
optimal governmental policy with overlapping generations includes
Diamond \cite{Diamond1}, Atkinson and Sandmo \cite{ AtkinsonSandmo},
Auerbach and Kotlikoff \cite{AuerbachKotlikoff}, Erosa and Gervais
\cite{ErosaGervais}.} However, our model ignores important aspects
of the governments policy tradeoffs that could result in a myriad of
alternative motives for capital income taxation. For example, the
Mirrlees approach to optimal taxation (see the papers surveyed in
Kocherlakota \cite{Kocherlakota1}) taught us that taxing capital
income can be required if the planner is facing an informational
friction due to the unobservable private skills or productivity. In
the closer (non Mirrleesian) context of overlapping generation,
Erosa and Gervais \cite{ErosaGervais} rationalize capital income tax
as an indirect leisure tax. Nonetheless, we aimed to argue that the
preferences based time inconsistency faced by the government by
itself creates a role for capital income taxation and we hope that the point is
clearest in our simple context.

The rest of the paper is organized as follow. The next section
presents the basic model and discuss the issue of time
inconsistency. We define the continuous time game in Section 3 and
provide the equilibrium characterization in Section 4. In Section 5
we undertake our application to the overlapping generations model.
The last section concludes.

\section{The model}

\label{Section: model}

\subsection{Preferences and production}

We consider a deterministic stationary environment where time is continuous. A
decision maker derives utility from a consumption schedule rate $c$, and the
date $t$ utility has the representation
\begin{equation}
\int_{t}^{\infty}h\left(  s-t\right)  u\left(  c\left(  s\right)  \right)  ds
\label{Utility}%
\end{equation}
for some utility function $u$ and some discount function $h$. We
assume that $u$ is strictly increasing, twice continuously
differentiable and strictly concave. The discount function $h$ is
continuously differentiable and positive, with $h\left(  0\right)
=1,$ $h\left(  t\right)  \geq0$, and
$\int_{0}^{\infty}h(s)ds<\infty$. Note that the discount function
between the current time $t$ and the consumption scheduling time $s$
depends only of $s-t$. This assumption implies, as in Strotz
\cite{Strotz}, that the incremental utility of immediate over
postponed consumption remains invariant with the passage of time. As
a result, the representation (\ref{Utility}) is
\emph{stationary}, meaning that
\[
\int_{t}^{\infty}h\left(  s-t\right)  u\left(  c\left(  s\right)  \right)
ds=\int_{0}^{\infty}h\left(  s\right)  u\left(  c\left(  t+s\right)  \right)
ds.
\]
The decision maker strives to maximize the objective (\ref{Utility}) under the
resource constraint
\begin{equation}
\frac{dk\left(  s\right)  }{ds}=f\left(  k\left(  s\right)  \right)  -c\left(
s\right)  ,\ \ k\left(  t\right)  =k_{t}, \label{Capital accumulation}%
\end{equation}
where $k(s)$ is capital at time $s$ and, where $f$ is a strictly increasing, concave and
continuously differentiable production function.

\subsection{Time consistency}

Unless the discount function is exponential, the marginal rate of consumption
substitution between two future dates will in general change with the mere
passage of time. To see this, fix the dates $t_{1}<t_{2}<t_{3}<t_{4}$ and the
consumption rates $c_{3},c_{4}$ and consider the marginal rate of substitution
(MRS) between consuming $c_{3}$ at date $t_{3}$ and consuming $c_{4}$ at date
$t_{4}$. When calculated from the perspective of date $t_{1}$, the MRS is
$\frac{h(t_{3}-t_{1})u^{\prime}(c_{3})}{h(t_{4}-t_{1})u^{\prime}(c_{4})}$
whereas the MRS is $\frac{h(t_{3}-t_{2})u^{\prime}(c_{3})}{h(t_{4}%
-t_{2})u^{\prime}(c_{4})}$ from the perspective of date $t_{2}$. The
two MRS will be identical if the discount function has the
multiplicative property
$\frac{h(t_{3}-t_{1})}{h(t_{4}-t_{1})}=\frac{h(t_{3}-t_{2})}{h(t_{4}-t_{2})}$
and this must hold for all $t_{1}<t_{2}<t_{3}<t_{4}$. As Strotz
\cite{Strotz} pointed out, a necessary and sufficient condition for
the MRS to be time invariant is that the discount function of the
exponential form $h(t)=e^{-\delta t}$ for some constant discount
rate $\delta\geq0$. The MRS changes with the mere passage of time
with any other discount function, and intertemporal consistency
fails. As a result, when the discount function is not of the
exponential form, a consumption schedule $\left(  c_{0}(s)\right)
_{s\geq0}$ seems optimal for a decision maker who maximizes the
objective (\ref{Utility}) under the constraint (\ref{Capital
accumulation}) at time $t=0$ and yet it will not be perceived as
such at later time $t>0$. So when time $t$ comes, there is no reason
to expect that the decision-maker will actually consume $c_{0}\left(
t\right)  $, as the decision-maker at time $0$ expected of her,
unless, of course, the latter has a way to commit the former. In
other words, for general discount functions, there are a plethora of
\textit{temporary} optimal policies: Each of them will be optimal
when evaluated from one particular point in time, but will cease to
be so when time moves forward. As a result, none of them can be
implemented, unless one of these viewpoints is given a privileged
status and the power to lock in policy for all future times (which,
incidentally, may be regretted afterwards).

In the absence of a commitment technology, the problem of maximizing
(\ref{Utility}) under the constraint (\ref{Capital accumulation}) can no
longer be seen as a classical optimization problem. There is no way for the
decision-maker at time $0$ to achieve what is, from her point of view, the
first-best solution of the problem, and she must turn to a second-best policy.
Defining and studying such a policy is the first aim of this paper. The path
to follow is clear and it is consistent with Strotz \cite{Strotz}, Pollak
\cite{Pollak} and Peleg and Yaari \cite{Peleg-Yaari}. The best the
decision-maker at time $t$ can do is to guess what her successors are planning
to do, and to plan her own consumption $c\left(  t\right)  $ accordingly. We
will then be looking for a subgame-perfect equilibrium of a certain game.

\section{Equilibrium strategies: construction and definition\label{sec1}}

We now proceed to define subgame-perfect equilibrium strategies. Paralleling
Aumann \cite{Aumann} we will consider the continuum of decision makers over
the time interval $[0,\infty)$. At any time $t$, there is a decision-maker who
decides what current consumption $c\left(  t\right)  $ shall be. As is readily
seen from the equation (\ref{Capital accumulation}), changing the value of $c$
at just one point in time will not affect the trajectory. However, the
decision-maker at time $t$ is allowed to form a coalition with her immediate
successors, that is with all $s\in\left[  t,\ t+\varepsilon\right]  $, and we
will derive the definition of an equilibrium strategy by letting
$\varepsilon\rightarrow0$. In fact, we are assuming that the decision-maker
$t$ can commit her immediate successors (but not, as we said before, her more
distant ones), but that the commitment span is vanishingly small. We now
construct and define the equilibrium Markov strategies. We analyze the problem
from the perspective of the decision maker at time $t=0$ but, given the
stationarity of the environment, a similar analysis can be carried out at any
time $t>0$.

We restrict our analysis to \emph{Markov} strategies, in the sense that the
policy depends only on a payoff relevant variable, the current capital stock
and not on past history, current time or some extraneous factors. Such a
strategy is given by $c=\sigma\left(  k\right)  $, where $\sigma:R\rightarrow
R$ is a continuously differentiable function. If we apply the strategy
$\sigma$, the dynamics of capital accumulation from $t=0$ are given by:
\[
\frac{dk}{ds}=f\left(  k\left(  s\right)  \right)  -\sigma\left(  k\left(
s\right)  \right)  ,\ k\left(  0\right)  =k_{0}%
\]

We shall say $\sigma$ \emph{converges to} $\bar{k}$, \textit{a steady state}
of $\sigma$, if $k\left(  s\right)  \longrightarrow\bar{k}$ when
$s\longrightarrow\infty$, when the initial value $k_{0}$ is sufficiently close
to $\bar k$. A strategy $\sigma$ is \emph{convergent} if there is some
$\bar{k}$ such that $\sigma$ converges to $\bar{k}$. In that case, the
integral (\ref{Utility}) is obviously convergent, and its successive
derivatives can be computed by differentiating under the integral. This
assumption is not required but it will greatly simplify the exposition, and
for this reason we will restrict our attention to convergent
strategies\footnote{In fact, we can work with a larger set of policies
$\sigma$ for which the integral (\ref{Utility}) is convergent, and for which
the successive derivatives of (\ref{Utility}) can be computed by
differentiating under the integral. The resulting equilibrium
characterizations of Section \ref{sec2} will be identical.}. Note that if
$\sigma$ converges to $\bar{k}$, then we must have $f\left(  \bar{k}\right)
=\sigma\left(  \bar{k}\right)  $.

Let us now proceed to the definition of equilibrium strategies. A convergent
Markov strategy $c=\sigma\left(  k\right)  $, where $\sigma:R\rightarrow R$ is
a continuously differentiable function, has been announced and is public
knowledge. The decision maker begins at time $t=0$ with capital stock $k$. If
all future decision-makers apply the strategy $\sigma$, the resulting capital
stock $k_{0}$ future path obeys
\begin{align}
\frac{dk_{0}}{dt}  &  =f\left(  k_{0}\left(  t\right)  \right)  -\sigma\left(
k_{0}\left(  t\right)  \right)  ,\ \ t\geq0\label{15}\\
k_{0}\left(  0\right)   &  =k. \label{16}%
\end{align}
We suppose the decision-maker at time $0$ can commit all the decision-makers
in $\left[  0,\ \varepsilon\right]  ,$where $\varepsilon>0$. She expects all
later ones to apply the strategy $\sigma$, and she asks herself if it is in
her own interest to apply the same strategy, that is, to consume
$\sigma\left(  k\right)  $. If she commits to another bundle, $c$ say, the
immediate utility flow during $\left[  0,\ \varepsilon\right]  $ is $u\left(
c\right)  \varepsilon$. At time $\varepsilon$, the resulting capital will be
$k+\left(  f\left(  k\right)  -c\right)  \varepsilon$, and from then on, the
strategy $\sigma$ will be applied which results in a capital stock $k_{c}$
satisfying
\begin{align}
&  \frac{dk_{c}}{dt}=f\left(  k_{c}\left(  t\right)  \right)  -\sigma\left(
k_{c}\left(  t\right)  \right)  ,\ \ t\geq\varepsilon\label{13}\\
&  k_{c}\left(  \varepsilon\right)  =k+\left(  f\left(  k\right)  -c\right)
\varepsilon. \label{14}%
\end{align}
The capital stock $k_{c}$ can be written as $k_{c}\left(  t\right)
=k_{0}\left(  t\right)  +k_{1}\left(  t\right)  \varepsilon$ where
\footnote{To see this, plug $k_{c}\left(  t\right)  =k_{0}\left(  t\right)
+k_{1}\left(  t\right)  \varepsilon$ into (\ref{13}) for $t\geq\varepsilon$,
keeping only the terms of first order in $\varepsilon$, and get
\begin{align*}
\frac{dk_{c}}{dt}  &  =f\left(  k_{0}\left(  t\right)  \right)  +\varepsilon
f^{\prime}\left(  k_{0}\left(  t\right)  \right)  k_{1}\left(  t\right)
-\sigma\left(  k_{0}\left(  t\right)  \right)  -\varepsilon\sigma^{\prime
}\left(  k_{0}\left(  t\right)  \right)  k_{1}\left(  t\right)  .
\end{align*}
Comparing this with (\ref{15}) gives (\ref{Ali1}). Equation (\ref{Ali2}) is
obtained by substituting the expansion $k_{0}\left(  \varepsilon\right)
=k+\varepsilon\frac{dk_{0}}{ds}\left(  0\right)  =k+\varepsilon\left(
f\left(  k\right)  -\sigma\left(  k\right)  \right)  $ into (\ref{14}).}
\begin{align}
&  \frac{dk_{1}}{dt}=\left(  f^{\prime}\left(  k_{0}\left(  t\right)  \right)
-\sigma^{\prime}\left(  k_{0}\left(  t\right)  \right)  \right)  k_{1}\left(
t\right)  ,\ t\geq\varepsilon\label{Ali1}\\
&  k_{1}\left(  \varepsilon\right)  =\sigma\left(  k\right)  -c \label{Ali2}%
\end{align}
where $f^{\prime}$ and $\sigma^{\prime}$ stand for the derivatives of $f$ and
$\sigma$. Summing up, we find that the total gain for the decision-maker at
time $0\,\ $from consuming bundle $c$ during the interval of length
$\varepsilon$ when she can commit, is
\[
u\left(  c\right)  \varepsilon+\int_{\varepsilon}^{\infty}h\left(  s\right)
u\left(  \sigma\left(  k_{0}\left(  t\right)  +\varepsilon k_{1}\left(
t\right)  \right)  \right)  dt,
\]
and in the limit, when $\varepsilon\rightarrow0$, and the commitment span of
the decision-maker vanishes, expanding this expression to the first order
leaves us with two terms
\begin{align}
&  \int_{0}^{\infty}h\left(  t\right)  u\left(  \sigma\left(  k_{0}\left(
t\right)  \right)  \right)  dt\nonumber\\
&  +\varepsilon\left[  u\left(  c\right)  -u\left(  \sigma(k)\right)
+\int_{0}^{\infty}h\left(  t\right)  u^{\prime}\left(  \sigma\left(
k_{0}\left(  t\right)  \right)  \right)  \sigma^{\prime}\left(  k_{0}\left(
t\right)  \right)  k_{1}\left(  t\right)  dt\right]  . \label{eq: total gain}%
\end{align}
where $k_{1}$ solves the linear equation
\begin{align}
\frac{dk_{1}}{dt}  &  =\left(  f^{\prime}\left(  k_{0}(t)\right)
-\sigma^{\prime}\left(  k_{0}(t)\right)  \right)  k_{1}\left(  t\right)
,\ \ t\geq0\label{RI}\\
k_{1}\left(  0\right)   &  =\sigma\left(  k\right)  -c. \label{RCI}%
\end{align}

Note that the first term of (\ref{eq: total gain}) does not depend on the
decision taken at time $0$, but the second one does. This is the one that the
decision-maker at time $0$ will try to maximize. In other words, given that a
strategy $\sigma$ has been announced and that the current state is $k\,$, the
decision-maker at time $0\,\ $faces the optimization problem:%
\begin{equation}
\max_{c}P_{1}\left(  k,\sigma,c\right)  \label{a14}%
\end{equation}
where
\begin{equation}
P_{1}\left(  k,\sigma,c\right)  =u\left(  c\right)  -u\left(  \sigma
(k)\right)  +\int_{0}^{\infty}h\left(  t\right)  u^{\prime}\left(
\sigma\left(  k_{0}\left(  t\right)  \right)  \right)  \sigma^{\prime}\left(
k_{0}\left(  t\right)  \right)  k_{1}\left(  t\right)  dt. \label{a141}%
\end{equation}
In the above expression, $k_{0}\left(  t\right)  $ solves the Cauchy problem
(\ref{15}),(\ref{16}) and $k_{1}\left(  t\right)  $ solves the linear equation
(\ref{RI}),(\ref{RCI}).

\begin{definition}
A convergent Markov strategy $\sigma:R\rightarrow R$ is an \emph{equilibrium
strategy for the intertemporal decision model} (\ref{Utility}) under the
constraint (\ref{Capital accumulation}) if, for every $k\in R$, the maximum in
problem (\ref{a14}) is attained for $c=\sigma\left(  k\right)  $:%
\begin{equation}
\sigma\left(  k\right)  =\arg\max_{c}P_{1}\left(  k,\sigma,c\right)  \label{9}%
\end{equation}

\end{definition}

The intuition behind this definition is simple. Each decision-maker can commit
only for a small time $\varepsilon,\,$\ so he can only hope to exert a very
small influence on the final outcome. In fact, if the decision-maker at time
$0$ plays $c$ when he/she is called to bat, while all the others are applying
the strategy $\sigma$, the end payoff for him/her will be of the form
\[
P_{0}\left(  k,\sigma\right)  +\varepsilon P_{1}\left(  k,\sigma,c\right)
\]
where the first term of the right hand side does not depend on $c$. In the
absence of commitment, the decision-maker at time $0$ will choose whichever
$c$ maximizes the second term $\varepsilon P_{1}\left(  k,\sigma,c\right)  $.
Saying that $\sigma$ is an equilibrium strategy means that the decision maker
at time $0$ will choose $c=\sigma\left(  k\right)  $. Given the stationarity
of the problem, if the strategy $c=\sigma\left(  k\right)  $ is chosen at time
$0$, it will be chosen at any future time $t$ and as a result, the strategy
$\sigma$ can be implemented in the absence of commitment.

Conversely, if a strategy $\sigma$ for the intertemporal decision
model (\ref{Utility}),(\ref{Capital accumulation}) is not an
equilibrium strategy, then it cannot be implemented unless the
decision-maker at time $0$ has some way to commit his successors.
Typically, an optimally committed strategy will not be an
equilibrium strategy. More precisely, a strategy which appears to be
optimal at time $0$ no longer appears to be optimal at times $t>0$,
which means that the decision-maker at time $t$ feels he can do
better than whatever was planned for him to do at time $0$. What
happens then if successive decision-makers take the myopic view, and
each of them acts as if he could commit his successors ?\ At time
$t$, then, the decision-maker would maximize the integral
(\ref{Utility}) with the usual tools of control theory, thereby
deriving a consumption $c=\sigma_{\mathrm{n}}\left(  t,k\right) $.
This is the naive strategy (O'Donoghue and Rabin \cite{O'Donoghue-Rabin1}); in general
it will not be an equilibrium strategy\footnote{Under non-constant
discounting, the commitment strategy $\sigma_{\mathrm{n}}$ is non
stationary and so we must extend the definition of the equilibrium
to a non stationary strategy. Although, we do not report the
definition of the non stationary equilibrium, it can easily be done
with some additional notations.}, so that every decision-maker has
an incentive to deviate.

\section{Characterization of the equilibrium strategies\label{sec2}}

\label{Section: Characterization}

The equilibrium strategy can be fully specified by a single function, the
\emph{value function} $v\left(  k\right)  $, which is reminiscent\ of-
although different from - the value function in optimal control. We will show
that the value function satisfies two equivalent equations, the integrated
equation (IE) and the differentiated equation (\ref{inf2}), the latter one resembling
the classical Hamilton-Jacobi-Bellman (HJB) equation of optimal control. This
similarity is reassuring since it shows how standard methods from control
theory can be adapted to analyze the impact of time inconsistency.
Unfortunately, the similarity is superficial only, since (\ref{inf2}) is a non-local
equation (and not a partial differential equation like (HJB)) and we will
demonstrate that its solutions exhibit different qualitative behavior. In the
knife edge case where the discount rate is constant, the non local term in
(\ref{inf2}) collapses, and (\ref{inf2}) becomes identical to (HJB). Consequently, when the
discount rate is constant, the equilibrium strategies are also optimal from
the perspective of all temporal decision makers.

Given a Markov strategy $\sigma\left(  k\right)  $, continuously
differentiable and convergent, we shall be dealing with the Cauchy problem
(\ref{15}), (\ref{16}). The value $k_{0}\left(  t\right)  $ depends on current
time $t$, initial data $k$, and the strategy $\sigma$. To stress this
dependence, it is convenient to write $k_{0}\left(  t\right)  =\mathcal{K}%
\left(  \sigma;t,k\right)  $ where $\mathcal{K}$ is the \emph{flow} associated
with the differential equation (\ref{15}) defined by
\begin{align}
\frac{\partial\mathcal{K}\left(  \sigma;t,k\right)  }{\partial t}  &
=f\left(  \mathcal{K}\left(  \sigma;t,k\right)  \right)  -\sigma\left(
\mathcal{K}\left(  \sigma;t,k\right)  \right) \label{eq: HJBid}\\
\mathcal{K}\left(  \sigma;0,k\right)   &  =k. \label{eq: HJBidic}%
\end{align}

The following theorem characterizes the equilibrium strategies and its proof
is given in Appendix \ref{A}. There are two parts in the equilibrium
characterization: a functional equation on the value function and an
instantaneous optimality condition determining current consumption. Here
$v^{\prime}$ is the derivative of $v$ and $i$ is the inverse of marginal
utility $u^{\prime}$. As usual, $i\circ v^{\prime}(x)=i\left(  v^{\prime
}\left(  x\right)  \right)  $

\begin{theorem}
\label{Th: condition necessaire}Let $\sigma:R\rightarrow R$ be a continuously
differentiable convergent strategy. If $\sigma$ is an equilibrium strategy,
then the value function
\begin{equation}
v\left(  k\right)  :=\int_{0}^{\infty}h\left(  t\right)  u\left(
\sigma\left(  \mathcal{K}\left(  \sigma;t,k\right)  \right)  \right)  dt
\label{valdef}%
\end{equation}
satisfies, for all $k$, the functional equation
\begin{equation}
v\left(  k\right)  =\int_{0}^{\infty}h\left(  t\right)  u\left(  i\circ
v^{\prime}\left(  \mathcal{K}\left(  i\circ v^{\prime};t,k\right)  \right)
\right)  dt \tag{IE}\label{inf1}%
\end{equation}
and the instantaneous optimality condition
\begin{equation}
u^{\prime}\left(  \sigma\left(  k\right)  \right)  =v^{\prime}\left(
k\right)  ,\ \ \sigma\left(  k\right)  =i\left(  v^{\prime}\left(  k\right)
\right)  \label{eq: enveloppe}%
\end{equation}

Conversely, if a function $v$ is twice continuously differentiable, satisfies
(\ref{inf1}), and the strategy $\sigma=i\circ v^{\prime}$ is convergent, then
$\sigma$ is an equilibrium strategy.
\end{theorem}

The instantaneous relation (\ref{eq: enveloppe}) expresses the usual tradeoff
between the utility derived from current consumption and the utility value of
saving. This is a standard condition in a world where there is one commodity
that can be used for investment or consumption. Let us spell out what equation
(\ref{inf1}) means. Given a candidate function $v$, we must first solve the
Cauchy problem (\ref{eq: HJBid}), (\ref{eq: HJBidic}) with $\sigma=i\circ
v^{\prime}$. Second, we calculate the right-hand side of equation
(\ref{inf1}), which is an integral along the trajectory of capital stock. The
final result should be equal to $v\left(  k\right)  $. Equation (\ref{inf1})
is therefore a fundamental characterization of the equilibrium strategies and
it takes the form of a functional equation on $v$. In order to contrast the
equilibrium dynamics with the dynamics resulting from using the optimal
control approach, the following proposition gives an alternative
characterization, the differentiated equation, which resembles the usual Euler
equation from optimal control and its proof is given in Appendix B.

\begin{proposition}
\label{Proposition: IE / DE} Let $v$ be a $C^{2}$ function such that the
strategy $\sigma=i\circ v^{\prime}$ converges to $\bar{k}$. Then $v$ satisfies
the integrated equation (IE) if and only if it satisfies the following
functional equation
\begin{equation}
-\int_{0}^{\infty}h^{\prime}(t)u\circ i\left(  v^{\prime}(\mathcal{K}\left(
i\circ v^{\prime};t,k\right)  \right)  dt=\sup_{c}\left[  u(c)+v^{\prime
}(k)\left(  f(k)-c\right)  \right] \label{inf2}%
\end{equation}
together with the boundary condition
\begin{equation}
v\left(  \bar{k}\right)  =u\left(  f\left(  \bar{k}\right)  \right)  \int
_{0}^{\infty}h\left(  t\right)  dt \label{BCChar}%
\end{equation}

\end{proposition}

It is useful to rewrite the differentiated equation (\ref{inf2}) as
\begin{equation}
\rho(k)=\frac{1}{v(k)}\sup_{c}\left[  u(c)+v^{\prime}(k)\left(  f(k)-c\right)
\right]  \label{dr}%
\end{equation}
where
\[
\rho(k)=-\frac{\int_{0}^{\infty}h^{\prime}\left(  t\right)  u\left(
\sigma\left(  \mathcal{K}\left(  \sigma;s,k\right)  \right)  \right)  dt}%
{\int_{0}^{\infty}h\left(  s\right)  u\left(  \sigma\left(  \mathcal{K}\left(
\sigma;s,k\right)  \right)  \right)  ds}%
\]
is interpreted as an effective discount rate. Equation (\ref{dr}) then tells
us that, along an equilibrium path, the relative changes in value to the
consumer must be equal to the effective discount rate. The effective discount
rate is here endogenous to the model and its presence reflects the strategic
behavior of the current decision maker resulting from internalizing the
behavior of future decision makers. In order to gain some insights into the
economic meaning of equation (\ref{inf2}), we first consider the exponential
discount function $h\left(  t\right)  =e^{-\delta_{0} t}$ for $\delta_{0}>0$.
With exponential discounting, $h^{\prime}(t)=-\delta_{0} h(t)$, and the
resulting effective discount rate is just the constant discount rate
$\rho(k)=\delta_{0}$ for all $k$. Equation (\ref{inf2}) becomes then simply
the familiar (HJB) equation
\begin{equation}
\delta_{0} v(k)=\sup_{c}\left[  u(c)+v^{\prime}(k)\left(  f(k)-c\right)
\right]  . \label{eq: HJB}%
\end{equation}
Second, we consider the case where $h$ is piecewise exponential,
$h(t)=e^{-\delta_{0}t}$ for $t\leq\tau$ and $h(t)=e^{-\delta_{1}t}$ for
$t>\tau$, with $\tau>0$ and $\delta_{0}>\delta_{1}$. The discount rate of this
discount function is decreasing with time and therefore the willingness to
postpone consumption at the margin is increasing over time. When the decision
maker is \textit{naive} in the sense that he acts as if he could commit the
future decision makers and does not learn from his past mistakes, his behavior
can be described by the (HJB) equation (\ref{eq: HJB}). In contrast, a
decision maker following the equilibrium strategy recognizes that future
decision makers will spend more than he currently hope and, in reaction to
that, he may accumulate more wealth than the naive decision maker. Equation
(\ref{inf2}) reflects exactly this idea since, when the discount function is
piecewise exponential, it can be written as
\begin{equation}
\delta_{0}v(k)=\sup_{c}\left[  u(c)+\left(  \delta_{0}-\delta_{1}\right)
g(k)+v^{\prime}(k)\left(  f(k)-c\right)  \right]  \label{eq: HJBexample}%
\end{equation}
where
\[
g(k)=\int_{\tau}^{\infty}e^{-\delta_{1}t}u\circ i\left(  v^{\prime
}(\mathcal{K}\left(  i\circ v^{\prime};t,k\right)  \right)  dt.
\]
The only difference between the naive policy characterization (\ref{eq: HJB})
and the equilibrium policy characterization (\ref{eq: HJBexample}) is the
extra term $(\delta_{0}-\delta_{1})g(k)$. Assuming that $g$ is positive and
increasing in $k$, it is then evident that the presence of the extra term
$(\delta_{0}-\delta_{1})g(k)$ yields additional incentives to save, relative
to the naive policy (where $g=0$).\footnote{To be more concrete, assume
further that the technology is linear, $f(k)=Ak$ and the utility of the form
$u(c)=\log(c)$. Solving (\ref{eq: HJB}) gives the naive policy
\[
\sigma_{n}(k)=\delta_{0}k
\]
whereas solving (\ref{eq: HJBexample}) gives the function
$g(k)=\frac {1}{\delta_{1}}e^{-\delta_{1}\tau}\log\left(  k\right) +
\varsigma$ where $\varsigma$ is a constant, and the resulting
equilibrium policy is
\[
\sigma_{e}(k)=\frac{\delta_{0}}{1+\frac{\delta_{0}-\delta_{1}}{\delta_{1}%
}e^{-\delta_{1}\tau}}k.
\]
This example illustrates the strategic motive of saving since the equilibrium
marginal propensity to consume is low relative to the naive policy.}

Neither equation (\ref{inf1}) nor equation (\ref{inf2}) are of a
classical mathematical type. If it were not for the integral term,
equation (\ref{inf2})\ would be a first-order partial differential
equation of known type (Hamilton-Jacobi), but this additional term,
which is\ non-local (an integral along the trajectory of the flow
(\ref{eq: HJBid}) associated with the solution $v$), creates a loss
of regularity in the functional equation that generates mathematical
complications. As a result, existence and uniqueness problems arise
as they typically do in dynamic games. The topic requires more
scrutiny. The next section applies the method in the context of an
overlapping generation model where time inconsistency is typically
faced by a utilitarian government.

\section{An overlapping generations growth model}

\subsection{The model}

\textit{Structure of the population.} We consider a continuous time
overlapping generations model of growth analysis, along the lines of
Blanchard (1985). The economy is composed of overlapping generations
of finitely-lived individuals who face a constant rate of death
$\pi>0$.\footnote{As mentioned by Blanchard \cite{Blanchard}, the
individual's rate of death can also be interpreted at the rate of
extinction of a dynasty. With this interpretation, the perpetual
youth assumption ($p$ constant) seems more acceptable. The
mathematical analysis suggests that the indeterminacy result does
not require the perpetual youth assumption.} At time $s$, the
probability of surviving until time $t\geq s$ is given by
$e^{-\pi(t-s)}$ and consequently, the expected life is
$\int_{s}^{\infty}t \pi e^{-\pi(t-s)} dt = \frac{1}{\pi}$. At each
instant a large cohort of identical individuals is born at a
normalized rate of $1$ so that the total number of individuals born
during a small time interval $[t_{1}, t_{2} )$ is $t_{2}-t_{1}$.
Because the cohort is large, there is no uncertainty on how the
cohorts's size and the total population vary over time. A cohort
born at time $\tau$ (the $\tau$-vintage) has a geometrically
declining size which, as of time $t\geq\tau$, is equal to
$e^{-\pi(t-\tau)}$. At each point of time $t$, the size of the
population is constant and it is given by
$\int_{-\infty}^{t}e^{-\pi(t-\tau)}d\tau=\frac{1}{\pi}$. The
time-$t$ expected lifetime utility of a vintage-$\tau$ individual
($\tau\leq t$) is, as in Yaari \cite{Yaari},
\[
\Gamma(\tau,t)=\int_{t}^{\infty}e^{-(\delta+\pi)(s-t)}\ln(c(\tau,s))ds,
\]
where $\delta>0$ is the constant pure rate of time preference and $c(\tau,s)$
is the consumption rate of an individual born at time $\tau$, as of time
$s\geq\tau$. The utility function of a newly born individual from
vintage$-\tau$ is then
\[
\Gamma(\tau,\tau)=\int_{\tau}^{\infty}e^{-(\delta+\pi)(s-\tau)}\ln
(c(\tau,s))ds.
\]

\textit{Technology.} The technology is represented by a constant return to
scale production function depending on two factors of production, aggregate
capital $K$ and aggregate labor. From above, the size of the population is
constant and assuming further that labor supply is inelastic, the production
function (net of depreciation) is a continuously
differentiable and concave function of aggregate capital stock $f(K)$. As in Section \ref{Section: model},
since capital and output are the same commodity, capital can be invested or
consumed and the investment rate is
\begin{equation}
\frac{dK(t)}{dt}=f(K(t))-C(t) \label{Capital accumulation bis}%
\end{equation}
where
\begin{equation}
C(t)=\int_{-\infty}^{t}c(\tau,t)e^{-\pi(t-\tau)}d\tau
\label{Aggregate consumption}%
\end{equation}
is the aggregate consumption at time $t$.

\textit{The social criterion.} We consider a social planner
maximizing a utilitarian criterion balancing the lifetime utilities
of the current population and the unborn generations. The planner is concerned with the generations' welfare from the
present (time $t_0=0$) onward and considers the alive individuals as
if they had just been born so that the criterion takes the form
\begin{equation}
\int_{0}^{\infty}e^{-\rho\tau}\Gamma(\tau,\tau)d\tau+\int_{-\infty}^{0}
e^{\pi\tau} \left(  \int_{0}^{\infty}e^{-(\delta+\pi)s}\ln(c(\tau
,s))ds\right)  d\tau. \label{eq: Criteria forward looking}
\end{equation}
The first term of the above criteria discounts back to time $0$ the
expected lifetime utility of unborn generations using the social
discount rate $\rho>0$. The second term is the remaining expected
lifetime utility of the individuals who were born in the past and
are still alive at time $0$. Notice the asymmetry of the treatment
of the unborn cohorts relative the the surviving ones in the
criteria (\ref{eq: Criteria forward looking}). The later's utilities
are discounted back to current time whereas the former are
discounted back to their birth date. Unlike the criteria of Calvo
and Obstfeld \cite{CalvoObstfeld} where symmetry is assumed, the
absence of symmetry in the criteria (\ref{eq: Criteria forward
looking}) creates a time inconsistency due to the dependency of the
planner's utility flow on the planning time. To see this, fix the
planning time $t_0=0$, change the order of integration in (\ref{eq:
Criteria forward looking}) and make the change of variable from
vintage $\tau$ to age $n=t-\tau$, to get the welfare function
\begin{align}
\int_{0}^{\infty} e^{-\rho s} \left\{  j_{u}(c,0,s) + j_{b}(c,0,s)
\right\}  ds.
\label{Welfare}%
\end{align}
where $j_{u}(c,0,s)$ is the time $s$ utility flow attributed to the
unborn cohorts defined for any $t \leq s$ by
\[
j_{u}(c,t,s) = \int_{0}^{s-t}e^{-(\delta+\pi - \rho)n}
\ln(c(s-n,s))dn,
\]
and where $j_{b}(c,0,s)$ is the time $s$ utility flow attributed to
the surviving cohorts defined for any $t \leq s$ by
\[
j_{b}(c,t,s)=\int_{s-t}^{\infty} e^{-(\delta+\pi -
\rho)n}e^{-(\delta-\rho)(s-t-n)}  \ln(c(s-n,s)) dn.
\]
Alternatively, if the planning time is $t_0=t>0$, a similar calculation shows that the planner's criterium becomes
\[
\int_{t}^{\infty} e^{-\rho (s-t)} \left\{  j_{u}(c,t,s) +
j_{b}(c,t,s) \right\}  ds.
\]
When $\delta = \rho$, time consistency obtains because the utility flow from the perspective of the planning times $t_0=t$ becomes $$j_{u}(c,t,s)+j_{b}(c,t,s)=\int_0^{\infty}
e^{- \pi n} \ln(c(s-n,s))dn$$ and therefore it is independent from the planning time. However, so long as $\delta > \rho$
the planner faces a time inconsistency problem because the utility flow $ j_{u}(c,t,s) +
j_{b}(c,t,s)$ depends explicitly on the planning time $t_0=t$.

\subsection{The centrally planned economy}

Beginning with a level of capital $K(0)$ at the planning time $t_0=0$, the
planner maximizes the criterion (\ref{Welfare}) under
the budget constraints (\ref{Capital accumulation bis}) and
(\ref{Aggregate consumption}). It is useful to
partition the planner's problem into two tasks each of which will
take the other as given. These two tasks are executed by two
fictitious planners, the {\it intra-period planner} and the {\it
metaplanner}.  The intra-period planner  takes as given the
aggregate consumption and is in charge of allocating the aggregate
consumption across the surviving cohorts. The metaplanner on the
other hand takes as given the path of intra cohorts' allocation
rules of aggregate consumption and is in charge of the aggregate
investment decision over time. Notice that both planners face a
dynamic decision problem. Let us now describe more formally the
task of the intra-period planner and the metaplanner.

At the planning time is $t_0=0$, the intra-period planners
commitment allocation of $C(s)$ at time $s\geq0$ is obtained by
maximizing
\[
\int_{0}^{s} e^{-(\delta+\pi-\rho)n} \ln(c(s-n,s))
dn+\int_{s}^{\infty }e^{-(\delta+\pi - \rho
)n}e^{-(\delta-\rho)(s-n)}\ln(c(s-n,s)) dn
\]
under the budget constraint $C(s)=\int_{0}^{\infty}c(s-n,s)e^{-\pi
n}dn$. The aggregate consumption expenditure $C(s)$ is exogenous to
the intra-period planner and the optimal commitment allocation is
\begin{align}
\label{eq: Commitment alloc rule 1}c(s-n,s)  &
=\pi(\delta+\pi-\rho)\frac{e^{(\rho-\delta)n}}{\pi+(\delta
-\rho)e^{-(\pi+\delta-\rho)s}}C(s),~~~~\text{ for }n\leq s,\\
\label{eq: Commitment alloc rule 2}c(s-n,s)  &
=\pi(\delta+\pi-\rho)\frac{e^{(\rho-\delta)s}}{\pi+(\delta
-\rho)e^{-(\delta+\pi-\rho)s}}C(s),~~~~\text{ for }n>s.
\end{align}
This allocation rule is however not robust to re-optimization because at the
planning time $t_0=t>0$, the intra-period planner's new commitment allocation is in general not aligned with the
allocation that the planner committed to at time $t_0=0$. For example, at the planning time $t_0=0$ the
intra-period planner allocation is
$$
c_{t_0=0}(t,t)=\pi(\delta+\pi-\rho) \frac{1}{\pi+(\delta
-\rho)e^{-(\pi+\delta-\rho)t}}C(t)
$$
for the cohort which is born at time $t$, but from the perspective
of the planning time $t_0=t$, the commitment allocation for the same cohort
is the egalitarian allocation
$$
c_{t_0=t}(t,t)= \pi C(t).
$$

In order to focus on a simple class of equilibria, we restrict our
analysis to the set of linear and stationary allocation rules for
the intra-period planner of the form
$$
c(t-n,t)=\varphi(n) C(t).
$$
with
\begin{equation} \label{eq: allocation rule constraint}
\int_0^{\infty} e^{- \pi n} \varphi(n)dn =1.
\end{equation}
Under this assumption, the objective (\ref{Welfare}) at the planning
time $t_0=0$ becomes
\[
\int_{0}^{\infty} e^{-\rho s} \left\{ \int_0^s
e^{-(\delta+\pi-\rho)n} \ln\left( \varphi(n) \right) dn
+\int_s^{\infty} e^{-\pi n} e^{-(\delta - \rho)s } \ln\left(
\varphi(n) \right) dn + L(s)\right\} ds,
\]
where $L(s)$ is a function depending only on time $s$ and the
aggregate consumption $C(s)$, which are exogenous and therefore, it drops out
of the intra-period planner's problem. Integrating by part this
formula, shows that the intra-period planner's objective at time
$t_0=0$ is given by
\begin{equation} \label{eq: intraperiod planner}
\left( \frac{1}{\rho} - \frac{1}{\delta} \right) \int_0^{\infty}
e^{-(\delta + \pi)s} \ln\left( \varphi(s)\right) ds +
\frac{1}{\delta} \int_0^{\infty} e^{-\pi s} \ln\left(
\varphi(s)\right) ds.
\end{equation}
When undertaken from the perspective of the planning time $t_0=t$,
the same calculations shows that the intra-period planner's
objective is identical with the objective (\ref{eq: intraperiod
planner}). Therefore, when restricted to stationary linear
allocation rules, all intra-period planners agree to use the rule
$\varphi( \cdot )$ that maximizes (\ref{eq: intraperiod planner})
under the constraint (\ref{eq: allocation rule constraint}). The
solution to this maximization problem gives the allocation rule
\begin{equation} \label{eq: optimal allocation rule}
\varphi(s) = \frac{\pi}{\delta} \frac{\delta + \pi}{\rho + \pi}
\left( (\delta - \rho) e^{-\delta s}+ \rho \right).
\end{equation}
How important is the assumption of stationary and linear allocation
rule? The assumption of stationarity of the allocation rule is
consistent with our view that the planner's problem in invariant
with the passage of time. The assumption of linearity is more
questionable because the equilibrium may require, for instance, to
have a less egalitarian consumption sharing rule in prosperous
periods. If we permit non linear allocation rules of the form
$c(t-n,t) = \varphi(n,C(t))$ we will in principle have to cope with
the time inconsistency of the intra-period planner, the time
inconsistency of the metaplanner (which, as we will show shortly,
arises even with linear allocation rules) and, moreover the
(sequential) interaction between the intra-period planner and the
metaplanner. We do not attempt to address this question fully.
Instead, assuming linear and stationary allocation rules allowed us
to work out a simple example of second best equilibrium where the
successive intra-period planners agree on the best allocation rule
and where the intra-period planner's problem is decoupled from the
metaplanner's problem.\footnote{Different assumptions on the
behavior of the intra-period planner are also possible and will lead
to an identical dynamic problem for the metaplanner. For example, if
we assume that the intra-period planners are naive, in the sense
that they do not internalize their time inconsistency problem, the
egalitarian policy allocation rule $c(t-n,t)=\pi C(t)$ will prevail.
The metaplanner's problem is then identical to the one that results
from using the allocation rule (\ref{eq: optimal allocation rule}).
The metaplanner's problem also remains intact if we assume that the
intra-period planner can commit to never change  the allocation rule
decided at time $t_0=0$. Under this assumption, the intra-period
planner's optimally committed allocation rule  is the non-stationary
and age dependent allocation rule (\ref{eq: Commitment alloc rule
1}), (\ref{eq: Commitment alloc rule 2}).}

Plugging either the optimal consumption allocation (\ref{eq: optimal
allocation rule}) (or the egalitarian allocation) in the planner's criterion (\ref{Welfare}) and
calculating the resulting integrals yields the time $s$ utility flow
to the metaplanner
\[
\frac{e^{-(\delta+\pi-\rho)s}}{\pi}\ln(C(s))+\frac{1-e^{-(\delta+\pi-\rho)s}%
}{\delta+\pi-\rho}\ln(C(s))+M(s)
\]
where $M(s)$ is a function depending only on the variables that are
exogenous and therefore, it drops out of the metaplanner's problem.
Substituting the utility flow into the criterion (\ref{Welfare}) and
dropping the function $M$ allows to express the metaplanner's
criterion at the planning time $t_0=0$ solely in terms of the
aggregate quantities
\begin{equation}
\underbrace{\int_{0}^{\infty}\frac{e^{-(\delta+\pi)s}}{\pi}\ln\left(
C(s)\right)  ds}_{\text{Welfare allocated to surviving cohorts}}%
+\underbrace{\int_{0}^{\infty}\frac{e^{-\rho
s}-e^{-(\delta+\pi)s}}{\delta +\pi-\rho}\ln\left(  C(s)\right)
ds}_{\text{Welfare allocated to unborn
cohorts}} \label{Welfare bis}%
\end{equation}
After rearranging and normalizing the welfare equation (\ref{Welfare
bis}), the metaplanner's objective at the planning time $t_0=0$ is
\begin{equation}
\int_{0}^{\infty}\frac{(\delta-\rho)e^{-(\delta+\pi)s}+\pi e^{-\rho s}}%
{\delta+\pi-\rho}\ln\left(  C(s)\right)  ds \label{Welfare ter}.
\end{equation}
A similar calculation shows that the metaplanner's objective from
the perspective of the planning time $t_0=t$ is
\begin{equation}
\int_{t}^{\infty}\frac{(\delta-\rho)e^{-(\delta+\pi)(s-t)}+\pi e^{-\rho (s-t)}}%
{\delta+\pi-\rho}\ln\left(  C(s)\right)  ds \label{Welfare
terlater}.
\end{equation}
The metaplanner is therefore facing a time consistency problem
because the marginal rates of substitution between consumption at
two dates in the future is changing by the mere passage of time.
This is exactly the type of time inconsistency that we discussed in the
sections $2$ to $4$. In the present context, time consistency is
endogenous since it is created from the structure of the criterion
(\ref{eq: Criteria forward looking}) which asymmetrically treats
the surviving cohorts (whose lifetime utility is discounted back to
current date) and the unborn cohorts (whose lifetime utility is
instead discounted back to birth date). Our planner is therefore
different from the time consistent planner of Calvo and Obstfled
(1998).

\subsection{Equilibrium strategies for the metaplanner}

As in Section \ref{Section: Characterization}, we denote by $v$ the
equilibrium value for the metaplanner. We are in the special case
when $u\left( c\right)  =\ln c$ and:
\begin{equation}
h\left(  t\right)
=\frac{\delta-\rho}{\pi+\delta-\rho}e^{-(\delta+\pi
)t}+\frac{\pi}{\pi+\delta-\rho}e^{-\rho t} \label{IEOG}%
\end{equation}

Due to the specific form of discounting in (\ref{IEOG}), it is possible to
rewrite the equilibrium characterization as a system of two differential
equations for two functions, the equilibrium value $v$, and a function $w$
which determines how the welfare is split between the surviving cohorts and
the unborn cohorts. The next proposition states this result and its proof is
given in the appendix.

\begin{proposition}
\label{prop1} Let $\sigma$ be a continuously differentiable strategy
converging to $\bar k$. If $\sigma$ is an equilibrium strategy, then the
functions $v$ and $w$ defined by
\begin{align}
v(k)  &  =\int_{0}^{\infty}\left(  \frac{\delta- \rho}{\pi+ \delta- \rho}
e^{-(\delta+\pi)t}+ \frac{\pi}{\pi+ \delta- \rho}e^{-\rho t}\right)
\ln\left(  \sigma\left(  \mathcal{K}\left(  \sigma;t,k\right)  \right)
\right)  dt,\label{eq1}\\
w(k)  &  =\frac{\pi}{\pi+ \delta- \rho} \int_{0}^{\infty}\left(
-e^{-(\delta+\pi)t}+e^{-\rho t}\right)  \ln\left(  \sigma\left(
\mathcal{K}\left(  \sigma;t,k\right)  \right)  \right)  dt \label{eq2}%
\end{align}
satisfy the system
\begin{align}
\left(  f-\frac{1}{v^{\prime}}\right)  v^{\prime}-\ln\left(  v^{\prime
}\right)   &  =\delta v-\left(  \delta-\rho\right)  w,\ \ \label{eq:ODE1}\\
\left(  f-\frac{1}{v^{\prime}}\right)  w^{\prime}  &  =-\pi v+\left(  \rho
+\pi\right)  w \label{eq:ODE2}%
\end{align}
with the boundary conditions
\begin{align}
v\left(  \bar{k}\right)   &  =\frac{\rho+\pi}{\rho\left(  \delta+\pi\right)
}\ln f\left(  \bar{k}\right) \label{BC1}\\
w\left(  \bar{k}\right)   &  =\frac{\pi}{\rho\left(  \delta+\pi\right)  }\ln
f\left(  \bar{k}\right)  \label{BC2}%
\end{align}
and the strategy $\sigma$ is given by $\sigma(k) = 1/v^{\prime}(k)$.

Conversely, let $v$ be a $C^{2}$ function such that the strategy
$\sigma(k)=1/v^{\prime}(k)$ converges to $\bar{k}$. If there exists a $C^{1}$
function $w$, such that $(v,w)$ satisfies the system (\ref{eq:ODE1}%
)-(\ref{eq:ODE2}) and the boundary conditions (\ref{BC1})-(\ref{BC2}), then
$\sigma$ is an equilibrium strategy converging to $\bar{k}$
\end{proposition}

Note that, by Proposition \ref{Proposition: IE / DE}, the
characterization (\ref{eq:ODE1})-(\ref{eq:ODE2}) is equivalent to
the more general equation (\ref{inf2}). Thus, when the discount function has
the special form (\ref{IEOG}), the non local one dimensional
equation (\ref{inf2}) becomes a system of two ordinary differential
equations without non local terms. This reduction is critical for
our existence result in the next theorem. The proof of
Proposition \ref{prop1} is given in the appendix.

Using the welfare decomposition (\ref{Welfare bis}), it can be shown that,
given the current level of capital $K(0)=k$, the criterion (\ref{Welfare ter})
allocates the welfare $v(k)-w(k)$ to the surviving cohorts whereas the unborn
cohorts' welfare is $w(k)$. To build some intuition on the system
(\ref{eq:ODE1}), (\ref{eq:ODE2}), it is useful to consider the case where
$\rho=\delta$, in which case equation (\ref{IEOG}) shows that the metaplanner is
facing a standard time consistent investment problem with a constant discount
rate $\delta$. As a result, the system $(v,w)$ is uncoupled in the sense that
$v$ can be solved for without knowing $w$. In this case, differentiating
(\ref{eq:ODE1}), gives the classical autonomous dynamical system describing
the evolution of the variables $(K,C)$
\begin{align*}
\frac{dK(t)}{dt}  &  =f(K(t))-C(t),\\
\frac{1}{C(t)}\frac{dC(t)}{dt}  &  =f^{\prime}(K(t))-\delta.
\end{align*}

When the metaplanner's problem is time inconsistent, the system $(K,C)$ is not
autonomous anymore and an additional variable $W(t)=w(K(t))$ is required to
describe the equilibrium dynamics. Using (\ref{eq:ODE1}), (\ref{eq:ODE2}) the
evolution of the economy can be described by the autonomous system $(K,C,W)$
\begin{align}
\label{eq:autonome1}\frac{dK(t)}{dt}  &  =f(K(t))-C(t),\\
\label{eq:autonome2}\frac{1}{C(t)}\frac{dC(t)}{dt}  &  =f^{\prime}(K(t))-\delta-\left(
\delta-\rho\right)  \frac{\pi}{\delta} \nonumber \\
&  -\frac{\delta-\rho}{\delta}\frac{C(t)}{f(K(t))-C(t)}\left(  \pi
\ln(C(t))-\rho\left(  \pi+\delta\right)  W(t)\right)  ,\\
\label{eq:autonome3}\frac{dW(t)}{dt}  &  =-\frac{1}{\delta}\left(  \pi\frac{f(K(t))-C(t)}%
{C(t)}+\pi\ln(C(t))-\rho\left(  \pi+\delta\right)  W(t).\right)
\end{align}

The dynamics of interest are the one where the economy naturally tends to a
steady state, that is a condition of the economy in which the aggregate level
of output, capital and consumption do not change over time. The following
theorem provides our main result on existence of multiple steady states

\begin{theorem}
\label{Th: Existence and multiplicity} Consider any $\bar{k} \in I$ where
$I=\{ k \mid\rho\frac{\pi+\delta}{\pi+\rho}<f^{\prime}(k)<\delta\}$. Then
there exist an equilibrium strategy, defined on some neighbourhood $\Omega$ of
$\bar{k}$ and converging to $\bar{k}$.
\end{theorem}

This result shows the existence of a continuum of steady states of
the economy. The result is proved in the appendix and it is obtained
by using the Central Manifold Theorem \cite{Carr}.
Associated to the multiple steady states is a continuum of
equilibrium strategies each of them generating a path of aggregate
capital stock. The equilibrium multiplicity is expectation driven in
the sense that if all metaplanners agree that $\bar{k}_{1}\in I$
then it will
and if all metaplanners agree that $\bar{k}_{2}\in I$ is a steady
state, then it will as well.
The multiplicity occurs because the subgame perfection condition is
not sufficient to pin down how the metaplanners coordinate their
beliefs on splitting the resources between the surviving cohorts and
the unborn ones.  When the social and private discount rate are
equal, the interval $I$ becomes one point and the economy's capital
stock converges then to its modified golden rule level $\bar k^{mg}$
defined by $f'\left( \bar k^{mg} \right) = \delta$. The next step is
then to evaluate the efficiency of these equilibria. In particular,
one may wonder if \textit{conservative} equilibrium policies (i.e.
those generating higher steady states level of capital and
consumption) Pareto dominate the less conservative equilibrium
policies. The distant unborn generations clearly favor the
conservative equilibria since the steady state consumption is
larger. However, the preference of the currently surviving cohorts
is ambiguous because current consumption must decrease in order to
achieve a higher steady state of capital and as a result, welfare
may be lower. While evaluating efficiency at this level of
generality does not seem obvious, equilibrium ranking is
nevertheless possible if one allow the metaplanners to revise their
belief at later stages of the game. So far, we implicitly ruled out
renegotiation in our definition of equilibrium policies since
beliefs were fixed once for all at an ex ante stage of the game.
This assumption may be justified when it is very costly to
renegotiate but it seems inappropriate in our context since it is
conceivable that the metaplanner may reconsider at future stages
their coordination with future metaplanners. Following an argument
initially introduced by Farell and Maskin \cite{Farrell} and
Bernheim and Ray \cite{Bernheim-Rey2} we illustrate in the next
subsection how giving the planner the ability to reconsider his
beliefs (present and future) - allows to shrink the set of subgame
perfect equilibria.

\subsection{Renegotiation-proof equilibrium}

In order to select a subgame perfect equilibrium, we use the notion
of renegotiation-proof equilibrium introduced by Farrell and Maskin
\cite{Farrell} in the context of a two players repeated game. The
notion has been used in a multiple selves infinite horizon game by
Kocherlakota \cite{Kocherlakota} (see also Asheim \cite{Asheim1} for
an alternative refinement) and we provide a local definition of this
notion.

\begin{definition}
An equilibrium strategy $\bar\sigma$ converging to $\bar k$ is locally
renegotiation-proof (l.r.p.) if there exists an open neighborhood $\Omega$ of
$\bar k$ such that for all $k^{*} \in\Omega$, the equilibrium policy
$\sigma^{*}$ converging to $k^{*}$ satisfies
\begin{equation}
\label{eq: lrp}\int_{0}^{\infty}h\left(  t\right)  u\left(  \sigma^{*}\left(
\mathcal{K}\left(  \sigma^{*};t,k_{0}\right)  \right)  \right)  dt\leq\int
_{0}^{\infty}h\left(  t\right)  u\left(  \bar{\sigma}\left(  \mathcal{K}%
\left(  \bar{\sigma};t,k_{0}\right)  \right)  \right)  dt
\end{equation}
for all $k_{0} \in\Omega$.
\end{definition}




This definition says that if an equilibrium policy is locally
renegotiation-proof then, from the perspective of the distant
metaplanners, a small perturbation of the policy, inducing a shift
in the steady state level of capital is dominated by the status quo.
To see this, notice that since the l.r.p. equilibrium policy
$\bar\sigma$ is converging to $\bar k$ then, for all $k_{0}$ in the
domain of definition of $\bar\sigma$, there exists a time $T>0$ such
that the flow $\left(  \mathcal{K}\left( \bar{\sigma};t,k_{0}\right)
, t \geq T \right)  $ belongs to $\Omega$. Therefore equation
(\ref{eq: lrp}) with $k_{0}=\mathcal{K}\left(
\bar{\sigma};s,k_{0}\right)  $ says that, for any $s \geq T$ the
metaplanner at time $s$ prefers to keep the strategy $\bar\sigma$
rather than switching to a neighboring rule $\sigma^{*}$. As a
result, if the metaplanners are restricted to switch policies by
small increments rather than by more drastic deviations, only l.r.p
equilibria are credible. The next proposition says that the local
renegotiation-proof condition rules out some subgame perfect
equilibria and allows to identify the steady state of the economy.

\begin{proposition}
\label{NL}Any equilibrium strategy $\bar{\sigma}$ converging to some
$\bar{k}$ satisfying
$\rho\frac{\pi+\delta}{\pi+\rho}<f^{\prime}(\bar{k})$ is not locally
renegotiation-proof: distant metaplanners will always agree to
switch to a new equilibrium policy converging to $k^{*}$ where
$k^{*}$ is slightly above $\bar k$.
\end{proposition}

The implementation of the switch can be done easily. Along an
equilibrium path converging to $\bar k$ and induced by the policy
$\bar\sigma$, the time $T$ metaplanner (where $T$ is sufficiently
large) suggest to the successors to switch to a new equilibrium
strategy $\sigma^{*}$ converging to $k^{*}$ where $k^{*}$ is
slightly larger that $\bar k$. Proposition \ref{NL} says that, for
$T$ sufficiently large, the metaplanners have an incentive to play
the new equilibrium policy $\sigma^{*}$ instead of keeping
$\bar{\sigma}$. In summary, all the future metaplanners (from time
$T$ onward) coordinate their beliefs on a new equilibrium
$\sigma^{*}$. Since the metaplanner at time $0$ is aware that this
will happen, it does not make sense for her to start playing
$\bar{\sigma}$ and as a result, she will use the strategy
$\sigma^{*}$. The same argument can however be made iteratively to
$\sigma^{*}$ and the future metaplanners will have interest to
deviate from $\sigma^{*}$ to a neighboring strategy ${\tilde
\sigma}$ converging to ${\tilde k}$ where ${\tilde k}$ is slightly
higher than $k^{*}$. The metaplanner will then apply an equilibrium
strategy converging to the highest steady state level of capital
defined by $f^{\prime}\left(  \bar k \right) =
\rho\frac{\pi+\delta}{\pi+\rho}$. Consequently, the steady state
level of capital depends on the planner's discount rate and the
individual discount rate. The more patient is the planner (lower
$\rho$), or the consumers (lower $\delta$), the higher is the steady
state level of aggregate capital stock stock. This result must be
contrasted with the Calvo and Obstfeld (1988) time consistent
government where the the private rate of time preference is
irrelevant for the steady state level of capital stock. Note that
when the planner's discount rate $\rho$ converges to $0$, the
capital stock converges to its golden rule level $\bar k^g$ defined
by $f' \left(\bar k^g \right) = 0$.

\subsection{Decentralization of the equilibria and fiscal policy}

The government begin at time $t_0=0$ with capital stock $K(0)$ and
tries to implement the second best allocation that we denote
$(K^*(t),C^*(t),W^*(t))_{t\geq 0}$ and such that the steady capital
stock is $K^*(\infty)=\bar k \in I$. This centralized allocation
path gives rise to the real interest rate $r^*_t=f'(K^*(t))$ and the
real per capita wage $\omega^*_t=(f\left( K^*(t) \right) - K^*(t)
f'\left( K^*(t) \right)) \pi$. We recall that  $(K^*,C^*,W^*)$
solves the autonomous system
(\ref{eq:autonome1}),(\ref{eq:autonome2}) and (\ref{eq:autonome3})
and that the resulting disaggregate allocation is given for any $\tau\leq t$ by
\begin{equation} \label{eq: disaggregate allocation}
c^*(\tau,t)=\varphi(t-\tau ) C^*(t)
\end{equation}
where the function $\varphi$ is defined by (\ref{eq: optimal allocation
rule}). Differentiating equation (\ref{eq: disaggregate allocation}) with respect to time and using
equation (\ref{eq:autonome2}), shows that the
the disaggregate allocation obeys
\begin{equation} \label{eq: disaggregate dynamic}
\frac{1}{c^*(\tau,t)} \frac{d\ c^*(\tau,t)}{dt} =r^*_t - \delta +
\psi(t) + \frac{\varphi'(t-\tau)}{\varphi(t- \tau)}
\end{equation}
where $\varphi'$ is the derivative of $\varphi$ and where
$$
\psi(t)=-\left( \delta-\rho \right) \frac{\pi}{\delta}
-\frac{\delta-\rho}{\delta}\frac{C^*(t)}{f(K^*(t))-C^*(t)}\left( \pi
\ln(C^*(t))-\rho\left(  \pi+\delta\right)  W^*(t)\right).
$$
When $\rho=\delta$, the intra-period planner uses the egalitarian
allocation rule and the metaplanner's problem admits a first best
solution since it is time consistent and we have $\varphi'=0$ and
$\psi=0$. As a result, when $\rho=\delta$, the disaggregate
consumption obeys $\frac{1}{c^*(\tau,t)} \frac{d\
c^*(\tau,t)}{dt}=r^*_t - \delta $.

The market economy is as the one of Yaari \cite{Yaari} and Blanchard
\cite{Blanchard}. In the absence of bequest motives, the mortality
risk creates a role for annuity contracts. A competitive insurance
market supplies the market with actuarially fair annuities that have
an instantaneous rate of return of $r_t + \pi$, where $r_t$ is the
market real interest rate as of time $t$. An individual born at time
$\tau$ who had accumulated a positive financial wealth $a(\tau , s)$
at time $s \geq \tau$  faces the risk of dying before spending it.
Investing in the annuity permits this individual to receive a
payment $(r_s + \pi) a(\tau, s)$ is he survived at time $s$ and, in
exchange, the insurance company will collect $a(\tau, s)$ if the
individual dies at time $s$. On the other hand, if the individual is
a net borrower, the lending institution faces the risk that the
individual dies before being able to repay the loan. The individual
will issue then an annuity on which a flow of interest of $(r_s +
\pi) a(\tau, s)$ is paid to the insurance company but in which the
debt will be forgiven if the issuer dies. Individuals will always
invest (or borrow) their financial wealth with the insurance
company: If they are lenders, the rate of return is higher with the
insurance company during their lifetime and if they are borrowers,
they have to insure themselves against death because negative
bequests are prohibited.

The fiscal instrument at the disposal of the government consists of
government, age dependent non distorting lump sum  taxation (that is
an income tax with no slope) and age dependent distorting capital
income taxes (or subsidy). Denoting by $a(\tau,t)$ the total assets
of the vintage $\tau$ agent at time $t$, the individual optimization
problem is
$$
\sup \int_t^{\infty} e^{(\delta +\pi) (s-t)} \log \left( c(\tau , s) \right) ds
$$
under the constraint
$$
\frac{d a(\tau,t)}{dt} = ((1 - \eta(\tau,t)) r_t + \pi) a(\tau,t)
-c(\tau,t) + \omega_t + \beta(\tau,t)
$$
where $\eta(\tau,t)$ is the tax rate at time $t$ for an agent born
at $\tau$ and, $\beta(\tau,t)$ is the transfer flow received by a
vintage $\tau$ agent at time $t$ and $r_t$ (resp. $\omega_t$) is the real interest
(resp. wage) rate prevailing at time $t$. The individual optimization problem has the first
order condition
\begin{equation} \label{eq: optimality condition}
\frac{1}{\hat c(\tau,t)} \frac{d\hat c(\tau,t)}{dt}
=(1-\eta(\tau,t))r_t - \delta
\end{equation}
for all $\tau \leq t$ and yields the policy
\begin{equation} \label{eq: optimality condition bis}
\hat c(\tau,t)= \left( \delta + \pi \right) \left[ a(\tau,t) +
h(\tau, t) + b(\tau, t) \right]
\end{equation}
where $h(\tau,t)$ is the human wealth of any individual born a time
$\tau$ as of time $t$,
$$
h(\tau,t) = \int_t^\infty \omega_s e^{-\int_t^s \left(
(1-\eta(\tau,x)) r_x + \pi\right) dx}ds
$$
and where $b(\tau,t)$ is the present value of government transfer flows expected by
a vintage $\tau$ agent as of time $t$,
\begin{equation} \label{eq: Present value lump sum}
b(\tau,t)=\int_t^\infty \beta(\tau,s) e^{-\int_t^s \left(
(1-\eta(\tau,x)) r_x + \pi\right) dx}ds.
\end{equation}
The problem of the government is to credibly commit to a fiscal
policy path $\{(\eta(\tau,t);\beta(\tau,t))_{ \tau \leq t} \}_{t
\geq 0} $, such that,  when the shadow prices $\{r^*_t\}_{t \geq 0}$
and $\{\omega^*_t\}_{t \geq 0}$ are expected, the individual optimal
consumption paths coincide with the desired allocation that is,
$$
\hat c (\tau,t) = c^* (\tau,t),~~
$$
for all $t \geq 0$ and $\tau \leq t$.

It is impossible to decentralize the allocation $c^*$
without government intervention because it is a second best
inefficient allocation. More precisely, in the absence of fiscal
intervention ($\eta=\beta=0$), summing at any date the individual
optimization policies over all surviving cohorts allows to identify
the aggregate dynamics. The resulting steady state of the market
economy is $\bar k_M$ where $\bar k_M$ is the unique solution of the
equation
$$
\left( f'(\bar k_M) - \delta \right) f \left( \bar k_M\right) = \pi (\delta + \pi) \bar k_M.
$$
We refer the reader to Section $1$ and $2$ of Blanchard
\cite{Blanchard} for the details. Blanchard \cite{Blanchard} showed
that that $f'\left( k_M \right)< \delta $ and, since $\delta < f'
\left( \bar k \right)$, the market economy does not sufficiently
accumulate the capital stock relative to the desired centrally
planned economy.

In order to decentralize  $c^*$, let us first assume that we can
find a fiscal policy $(\eta, \beta)$ such that the resulting
individual consumption matches the desired allocation at time $0$,
that is, $\hat c (\tau , 0) = c^*(\tau, 0)$ for all $\tau \leq 0$.
If we want the decentralization to follow through for later
times $t>0$, we must check that $\hat c (\tau , t) = c^*(\tau, t)$
for any $\tau<0$ and $t>0$. This condition will hold if and only if
the optimality condition (\ref{eq: optimality condition}) is
identical to the disaggregate allocation dynamics equation (\ref{eq:
disaggregate dynamic}) which results in
$$
(1-\eta(\tau,t))r^*_t - \delta = r^*_t - \delta + \psi(t) +
\frac{\varphi'(t-\tau)}{\varphi(t- \tau)}
$$
This condition can be expressed as
\begin{equation} \label{eq: tax rate on capital}
\eta(t-n , t)= \frac{1}{r^*_t} \left( -\psi(t)  -
\frac{\varphi'(n)}{\varphi(n)} \right)
\end{equation}
where $n=t-\tau$ is the age of the taxable consumer. Equation (\ref{eq: tax rate on capital})
uniquely identifies the required capital income tax rate since it is
the only distortional tax rate prompting the necessary credibility of the taxation policy.

When the planner's discount rate is equal to the individual discount
rate; $\rho=\delta$, the time inconsistency disappears, the
centrally planned allocation becomes first best and the capital
income tax rate is zero ($\psi=\varphi'=0$). When $\delta \neq
\rho$, the tax policy is age dependent. However, if the intra-period
planner uses the egalitarian allocation rule instead of the
allocation rule (\ref{eq: optimal allocation rule}), the tax rate
becomes age independent and it given by $\eta( t)= -  \psi(t)/
r^*_t$.

At date $t=0$, the government has no flexibility in the choice of
the tax rate policy $\eta$ because it is given by equation (\ref{eq:
tax rate on capital}) but it will use the lump sum taxation in order
to match the initial individual consumption $\hat c(\tau,0)$ with
the disaggregate allocation $c^*(\tau,0)$. To see this, use the
policy formula (\ref{eq: optimality condition bis}) for an agent
born at time $\tau=0$ to get
$$
\hat c(0,0)= \left( \delta + \pi \right) \left[ a(0,0) + h(0, 0) +
b(0, 0) \right]
$$
and, assuming that  $a(0,0)=0$; people inherit neither capital stock
nor debt when they are born, we see that $\hat c(0,0) = c^*(0,0)$ if
and only if
\begin{equation} \label{eq: lump sum tax1}
b(0,0)= \frac{1}{\delta + \pi} c^*(0,0) - h(0,0).
\end{equation}
Equation (\ref{eq: optimality condition bis}) also gives a similar
formula for the present value of transfer flows to the consumers
born at $\tau<0$,
\begin{equation} \label{eq: lump sum tax2}
b(\tau,0)= \frac{1}{\delta + \pi} c^*(\tau,0) - h(\tau,0) -
a(\tau,0)
\end{equation}
where $a(\tau,0)$ is the total assets (capital stock and government
debt) of a consumer born at time $\tau$ as of time $0$.

Finally, the government will face the same decentralization problem
for the generations that will be born at $\tau>0$ and equation (\ref{eq: optimality
condition bis}) gives again the formula
\begin{equation} \label{eq: lump sum tax3}
b(\tau,\tau)= \frac{1}{\delta + \pi} c^*(\tau,\tau) - h(\tau,\tau).
\end{equation}
The right hand sides of (\ref{eq: lump sum tax1}), (\ref{eq: lump
sum tax2}) and (\ref{eq: lump sum tax3}) are known to the government whereas the left
hand sides are related to the transfers flow $\beta$ through the
integral formula (\ref{eq: Present value lump sum}). Any lump sum
transfers flow $\beta$ satisfying (\ref{eq: lump sum tax1}),
(\ref{eq: lump sum tax2}) and (\ref{eq: lump sum tax3})
decentralizes the desired plan and hence, unlike the capital income
taxation, the lump sum taxation is indeterminate.

Notice that the fiscal policy that decentralizes the plan $c^*$ is time consistent: even if
the government is given the opportunity to revise the path of fiscal instruments over time,
it has no interest in doing so in the absence of new information because the problem of time inconsistency is sorted out at the
{\it ex ante } stage of central planning.

Evaluating equation (\ref{eq:autonome2}) at the steady state of the
economy gives $\lim_{t \nearrow \infty} \psi(t)=\bar \psi = \delta -
f'(\bar k)
>0$ and therefore, the steady state tax rate becomes stationary and
it is given by
\begin{equation} \label{eq: capital income tax LT}
\lim_{t \nearrow \infty} \eta(t-n,t)  =\frac{1}{f'\left(\bar k
\right)} \left( f'\left(\bar k \right) -\delta -
\frac{\varphi'(n)}{\varphi(n)} \right).
\end{equation}

This result shows that it is required to tax (or subsidize) capital
income in the long term and it is in stark contrast with
representative infinitely lived agents models where the capital
interest taxation is not an optimal instrument in the long term.

The long term capital income tax rate (\ref{eq: capital income tax LT})
require to subsidize the old cohorts and tax the young cohorts. In
fact, it can be shown that there exist a threshold age $\tilde n$
that the government subsidizes (resp. taxes) capital income for
cohorts older (resp. younger) than $\tilde n$ years. For instance, if the planner decide to decentralize the l.r.p.
allocation converging to $\bar k$ with $f'(\bar k) = \rho\frac{\pi+\delta}{\pi+\rho}$, the cutoff age is given by
$$
\tilde n = - \frac{1}{\delta} \ln \left( \frac{\pi}{\delta + \pi} \right)
$$

Notice that if the intra period planner decides to use the
egalitarian allocation $\varphi (n) = \pi$, there is no need for an
age dependent capital income taxation. Under this assumption, the
government will simply use the long term uniform capital income
subsidy rate
$$
\bar \eta = \frac{\pi}{ \rho} \frac{\delta - \rho}{\pi + \rho}
$$
and finance this subsidy by lump sum transfers and/or by issuing bonds.

We must emphasize that if the government has access to a commitment
technologies, it will adopt an allocation consistent with the
asymptotic discount rate of the discount function (\ref{IEOG}), that
is $\rho$. The result of Calvo and Obstfeld \cite{CalvoObstfeld}
suggests then that long term capital income taxation is not
required. Therefore, in our model, the capital income taxation is
entirely driven by the time inconsistency friction faced by the
government.

\section{Conclusion\label{sec5}}

The central concern of this paper has been the attempt to define a
methodology to analyze a capital accumulation dynamic game with non
constant discount in continuous time and under the assumption of non
commitment. The methodology only requires standard differential
calculus techniques and the resulting equilibrium characterization
is a generalization of the HJB equation. Using a special fixed point
theorem, the central manifold theorem (Carr \cite{Carr}), we were
able to prove existence of multiple equilibria for a specification
of the discount function. We perceive the connection of the
methodology with the central manifold theorem as a good news: In
fact the central manifold theorem comes with an approximation method
(See Theorem 2 of Carr \cite{Carr}) that opens the door to
computational policy experimentation. Taken overall, the method
seems to be applicable to a broad set of time inconsistency
problems. We consider an application to the dynamic allocation
problem for a forward looking utilitarian government in an
overlapping generations economy. When the social discount rate and
the private discount rate are distinct, the optimal command is time
inconsistent and the government becomes strategic. We find that
there are multiple steady states for the economy but that the
(local) renegotiation-proofness requirement selects one of them. The
strategic allocation can be implemented in a market economy if the
government deploys distortional capital income taxation. Our main
point here is that the time inconsistency friction that arises from
the planning problem creates by itself a role for capital income
taxation. The analysis of our overlapping generations model has been
limited in several important respects. In particular, we did not
impose any restriction on the fiscal tools, neither directly as in
the Ramsey approach of optimal taxation nor indirectly as in the
Mirrlees approach to optimal taxation. The non restrictiveness of
the fiscal tools allowed us to achieve a separation between the
allocation problem and the taxation problem. As a result, we shut
down the time inconsistency generated by the interaction between the
government and the private agents. An unresolved question that we
plan to investigate in future work is to see how the preferences
based time inconsistency interacts with the time inconsistency due
instruments insufficiency. This can be done in the context of the
Mirrleesian approach to optimal taxation and the main difficulty is
to model a game with multiple private agents and multiple successive
governments.

\appendix

\section{Proof of Theorem \ref{Th: condition necessaire}\textit{\label{A}}}

\subsection{Preliminaries}

Before proceeding with the proof of the theorem, let us mention some facts
about the flow $\mathcal{K}$ defined by (\ref{eq: HJBid}), (\ref{eq: HJBidic}%
). To make notations simpler, we will use $\mathcal{K}\left(  t,k\right)  $
instead of $\mathcal{K}\left(  \sigma; t,k\right)  $ when the omission of the
dependency on $\sigma$ causes no ambiguity.

Note first that, since equation (\ref{eq: HJBid}) is autonomous, i.e. the
right-hand side does not depend explicitly on time, the solution which takes
the value $k$ at time $0$ coincides with the solution which takes the value
$\mathcal{K}\left(  t,k\right)  $ at time $t\geq0.$ This is the so-called
semi-group property, which is stated precisely as follows
\begin{equation}
\mathcal{K}\left(  s,\mathcal{K}\left(  t,k\right)  \right)  =\mathcal{K}%
\left(  s+t,k\right)  \label{Rdf}%
\end{equation}

Next, consider the linearized equation around a prescribed solution
$t\rightarrow\mathcal{K}\left(  t,k\right)  $ of the nonlinear system
(\ref{eq: HJBid}), namely:%

\begin{equation}
\frac{dk_{1}}{ds}=\left(  f^{\prime}\left(  \mathcal{K}\left(  s,k\right)
\right)  -\sigma^{\prime}\left(  \mathcal{K}\left(  s,k\right)  \right)
\right)  k_{1}\left(  s\right)  \label{a26}%
\end{equation}

This is a linear equation, so the flow is linear. The value at time $t$ of the
solution which takes the value $k$ at time $0\ $is $\mathcal{R}(t)k$, where
the function $\mathcal{R}:R \rightarrow(0,\infty)$ satisfies:
\begin{align}
\frac{d\mathcal{R}}{dt}  &  =\left(  f^{\prime}\left(  \mathcal{K}\left(
s,k\right)  \right)  -\sigma^{\prime}\left(  \mathcal{K}\left(  s,k\right)
\right)  \right)  \mathcal{R}\left(  t\right) \label{a31}\\
\mathcal{R}\left(  0\right)   &  =1. \label{a32}%
\end{align}

From standard theory, it is well known that, if $f$ and $\sigma$ are $C^{k}$,
then $\mathcal{K}$ is $C^{k-1}$, and:
\begin{align*}
\frac{\partial\mathcal{K}\left(  t,k\right)  }{\partial k}  &  =\mathcal{R}%
\left(  t\right) \\
\frac{\partial\mathcal{K}\left(  t,k\right)  }{\partial t}  &  =-\mathcal{R}%
(t)\left(  f\left(  k\right)  -\sigma\left(  k\right)  \right)  .
\end{align*}

Let us now turn to the actual proof of Theorem \ref{Th: condition necessaire}.

\subsection{Necessary condition}

Given a Markov equilibrium strategy $\sigma$, we define the associated value
function $v(k)$ as in formula (\ref{valdef})
\begin{equation}
v\left(  k\right)  :=\int_{0}^{\infty}h\left(  t \right)  u\left(
\sigma\left(  \mathcal{K}\left(  t,k\right)  \right)  \right)  dt \label{Vf}%
\end{equation}

Differentiating with respect to $k$, we find that:%
\begin{align*}
v^{\prime}\left(  k\right)   &  =\int_{0}^{\infty}h(t)u^{\prime}\left(
\sigma(\mathcal{K}\left(  t,k \right)  )\right)  \sigma^{\prime}\left(
\mathcal{K}\left(  t,k \right)  \right)  \frac{\partial\mathcal{K}}{\partial
k}\left(  t,k \right)  dt\\
&  =\int_{0}^{\infty}h(t)u^{\prime}\left(  \sigma(\mathcal{K}\left(  t,k
\right)  )\right)  \sigma^{\prime}\left(  \mathcal{K}\left(  t,k \right)
\right)  \mathcal{R}(t)dt
\end{align*}

Since $\sigma$ is an equilibrium strategy, the maximum of $P_{1}\left(
k,\sigma,c\right)  $ with respect to $c$ must be attained at $\sigma\left(
k\right)  $. The function $P_{1}$ itself is given by formula (\ref{a141}),
where $k_{0}(t) = \mathcal{K}\left(  t,k \right)  $ and where $k_{1}$ is
defined by (\ref{RI}) and (\ref{RCI}), so that $k_{1}(t)=\mathcal{R}(t)\left(
\sigma(k)-c\right)  $. Substituting into (\ref{a141}), we get:
\begin{align*}
P_{1}\left(  k,\sigma,c\right)   &  =u\left(  c\right)  -u\left(
\sigma(k)\right) \\
&  +\int_{0}^{\infty}h\left(  t\right)  u^{\prime}\left(  \sigma\left(
\mathcal{K}\left(  t,k \right)  \right)  \right)  \sigma^{\prime}\left(
\mathcal{K}\left(  t,k \right)  \right)  \mathcal{R}\left(  t\right)  \left(
\sigma(k)-c\right)  dt
\end{align*}

Since $u$ is concave and differentiable, the necessary and sufficient
condition to maximize $P_{1}(k,\sigma,c)$ with respect to $c$ is
\[
u^{\prime}(c)=\int_{0}^{\infty}h(t)u^{\prime}\left(  \sigma(\mathcal{K}\left(
t,k\right)  )\right)  \sigma^{\prime}\left(  \mathcal{K}\left(  t,k\right)
\right)  \mathcal{R}(t)dt
\]
which is precisely $v^{\prime}(t,k)$, as we just saw. Therefore, the
equilibrium strategy must satisfy
\[
u^{\prime}\left(  \sigma(k)\right)  =v^{\prime}(k)
\]
and, substituting back into equation (\ref{Vf}), we get equation (\ref{inf1}).

\subsection{Sufficient condition}

Assume now that there exists a function $v$ satisfying (\ref{inf1}), and
consider the strategy $\sigma=i\circ v^{\prime}$. Given any consumption choice
$c\in R$, the payoff to the decision-maker at time $0$ is
\begin{align*}
P_{1}\left(  k,\sigma,c\right)   &  =u\left(  c\right)  -u\left(
\sigma(k)\right) \\
&  +\int_{0}^{\infty}h\left(  t\right)  u^{\prime}\left(  \sigma\left(
\mathcal{K}\left(  t,k \right)  \right)  \right)  \sigma^{\prime}\left(
\mathcal{K}\left(  t,k \right)  \right)  \mathcal{R}\left(  t\right)  \left(
\sigma(k)-c\right)  dt\\
&  =u\left(  c\right)  -u\left(  \sigma(k)\right)  +v^{\prime}(k)\left(
\sigma\left(  k\right)  -c\right) \\
&  =u\left(  c\right)  -u\left(  \sigma(k)\right)  -u^{\prime}\left(
\sigma\left(  k\right)  \right)  \left(  c-\sigma\left(  k\right)  \right) \\
&  \leq0,
\end{align*}
where the first equality follows from the definition of $\mathcal{R}$, the
second equality is obtained by differentiating $v$ with respect to $k$, the
third equality follows from the definition of $\sigma$, and the last
inequality is due to the concavity of $u$. Observing that $P_{1}%
(k,\sigma,\sigma\left(  k\right)  )=0$, we see that the inequality
$P_{1}(k,\sigma,c)\leq0$ implies that $c=\sigma\left(  k\right)  $ achieves
the maximum so that $\sigma$ is an equilibrium strategy.

\section{Proof of Proposition \ref{Proposition: IE / DE}\label{B}}

Let a function $v:R \rightarrow R $ be given. Consider the function$\varphi:R
\rightarrow R $ defined by
\begin{equation}
\varphi\left(  k\right)  =v\left(  k\right)  -\int_{0}^{\infty}h(t)u\left(
\sigma\left(  \mathcal{K}\left(  \sigma;t,k\right)  \right)  \right)  dt
\label{eq6}%
\end{equation}
where $\sigma\left(  k\right)  =i\left(  v^{\prime}\left(  k\right)  \right)
$. Consider $\psi\left(  t,k\right)  $, the value of $\varphi$ along the
trajectory $t\longrightarrow\mathcal{K}\left(  \sigma;t,k\right)  $
originating from $k$ at time $0$, that is
\begin{align*}
\psi\left(  t,k\right)   &  =\varphi\left(  \mathcal{K}\left(  t,k\right)
\right) \\
&  =v\left(  \mathcal{K}\left(  t,k\right)  \right)  -\int_{0}^{\infty
}h(s)u\left(  \sigma(\mathcal{K}\left(  s,\mathcal{K}\left(  t,k\right)
\right)  )\right)  ds\\
&  =v\left(  \mathcal{K}\left(  t,k\right)  \right)  -\int_{0}^{\infty
}h(s)u\left(  \sigma(\mathcal{K}\left(  s+t,k\right)  )\right)  ds\\
&  =v\left(  \mathcal{K}\left(  t,k\right)  \right)  -\int_{t}^{\infty
}h(s-t)u\left(  \sigma(\mathcal{K}\left(  s,k\right)  )\right)  ds
\end{align*}
where we have used formula (\ref{Rdf}).

We compute the derivative of this function with respect to $t$:%
\begin{align*}
\frac{\partial\psi}{\partial t}\left(  k,t\right)   &  =v^{\prime}\left(
\mathcal{K}\left(  t,k\right)  \right)  \left[  f\left(  \mathcal{K}\left(
t,k\right)  \right)  -i\left(  \sigma\left(  \mathcal{K}\left(  t,k\right)
\right)  \right)  \right] \\
&  +u\left(  \sigma\left(  \mathcal{K}\left(  t,k\right)  \right)  \right)
+\int_{t}^{\infty}h^{\prime}(s-t)u\left(  \sigma(\mathcal{K}\left(
s,k\right)  )\right)  ds
\end{align*}

From the definition of $i$, we have
\[
u\left(  i\left(  v^{\prime}\left(  \mathcal{K}\left(  t,k\right)  \right)
\right)  \right)  -v^{\prime}\left(  \mathcal{K}\left(  t,k\right)  \right)
i\left(  v^{\prime}\left(  \mathcal{K}\left(  t,k\right)  \right)  \right)
=\sup_{c}\left\{  u\left(  c\right)  -v^{\prime}\left(  \mathcal{K}\left(
t,k\right)  \right)  c \right\}
\]

Substituting in the preceding equation, and recalling that $\sigma=i\circ
v^{\prime}$ gives
\begin{align*}
\frac{\partial\psi}{\partial t}\left(  k,t\right)   &  =\sup_{c}\left\{
u\left(  c\right)  +v^{\prime}\left(  \mathcal{K}\left(  t,k\right)  \left(
f\left(  \mathcal{K}\left(  t,k\right)  \right)  -c\right)  \right)  \right\}
+\int_{t}^{\infty}h^{\prime}(s-t)u\left(  \sigma(\mathcal{K}\left(
s,k\right)  )\right)  ds\\
&  = \sup_{c}\left\{  u\left(  c\right)  +v^{\prime}\left(  \mathcal{K}\left(
t,k\right)  \left(  f\left(  \mathcal{K}\left(  t,k\right)  \right)
-c\right)  \right)  \right\}  +\int_{0}^{\infty}h^{\prime}(s)u\left(
\sigma(\mathcal{K}\left(  s, \mathcal{K}\left(  t,k\right)  \right)  )\right)
ds
\end{align*}
where the second equality is obtained by using a change of variable and
formula (\ref{Rdf}). If (\ref{inf2}) holds, then the right-hand side of the last
equation is identically zero along the trajectory, so that $\psi\left(
k,t\right)  =\psi\left(  k\right)  $ does not depend on $t$. Letting
$t\longrightarrow\infty$ in the definition of $\psi$, we get:%
\begin{align*}
\psi\left(  k\right)   &  =\lim_{t\longrightarrow\infty}\left\{  v\left(
\mathcal{K}\left(  t,k\right)  \right)  -\int_{0}^{\infty}h(s)u\left(
\sigma(\mathcal{K}\left(  s+t,k\right)  )\right)  ds\right\} \\
&  =v\left(  \bar{k}\right)  -\int_{0}^{\infty}h\left(  s\right)  u\left(
\sigma\left(  \bar{k}\right)  )\right)  ds =v\left(  \bar{k}\right)  -
u\left(  f \left(  \bar{k}\right)  \right)  \int_{0}^{\infty}h\left(
s\right)  ds
\end{align*}
and hence, if (\ref{BCChar}) holds then, $\psi=\varphi=0$ and equation (\ref{inf1}) holds.

Conversely, if $v\left(  k\right)  $ satisfies equation (\ref{inf1}), then the
same lines of reasoning shows that equation (\ref{inf2}) and the boundary condition
are satisfied.

\section{Proof or Proposition \ref{prop1}}

The proof uses the following:

\begin{lemma}
Let $\sigma\left(  k\right)  $ be any convergent Markov strategy. Denote its
steady state by $\bar{k}$. Let $h_{0}:\left[  0,\ \infty\right]
\longrightarrow R$ be any $C^{1}$ function with exponential decay at infinity,
that is
\[
h_{0}(t) \leq C e^{- \nu t}
\]
for some positive constants $\nu>0$ and $C\geq0$. A $C^{1}$ function $I$
satisfies
\begin{equation}
I^{\prime}\left(  k\right)  \left(  f\left(  k\right)  -\sigma\left(
k\right)  \right)  +\int_{0}^{\infty}h_{0}^{\prime}\left(  t\right)
\ln\left(  \sigma\left(  \mathcal{K}\left(  \sigma;t,k\right)  \right)
\right)  dt+h_{0}\left(  0\right)  \ln\left(  \sigma\left(  k\right)  \right)
=0 \label{eq11}%
\end{equation}
for all $k$ and, the boundary condition
\begin{equation}
I\left(  \bar{k}\right)  =\int_{0}^{\infty}h_{0}\left(  t\right)  \ln f\left(
\bar{k}\right)  \label{eq11BC}%
\end{equation}
if and only if it satisfies
\begin{equation}
I(k)=\int_{0}^{\infty}h_{0}\left(  t\right)  \ln\left(  \sigma\left(
\mathcal{K}\left(  \sigma;t,k\right)  \right)  \right)  dt. \label{eq10}%
\end{equation}

\end{lemma}

\begin{proof}
We argue as in the preceding proof. For any $C^{1}$ function $I$, consider the
function $\psi\left(  k,t\right)  $ defined by
\[
\psi\left(  \sigma; k,t\right)  =I\left(  \mathcal{K}\left(  \sigma
;t,k\right)  \right)  -\int_{t}^{\infty}h_{0}(s-t) \ln\left(  \sigma
(\mathcal{K}\left(  \sigma;s,k\right)  )\right)  ds.
\]

Differentiating with respect to $t$, we get%
\begin{align*}
\frac{\partial\psi}{\partial t}  &  =I^{\prime}\left(  \mathcal{K}\left(
\sigma;t,k\right)  \right)  \left(  f\left(  \mathcal{K}\left(  \sigma
;t,k\right)  -\sigma\left(  \mathcal{K}\left(  \sigma;t,k\right)  \right)
\right)  \right) \\
&  +\int_{t}^{\infty}h_{0}^{\prime}(s-t) \ln\left(  \sigma(\mathcal{K}\left(
\sigma;s,k\right)  )\right)  ds+h_{0}\left(  0\right)  \ln\left(
\sigma(\mathcal{K}\left(  \sigma;t,k\right)  ) \right)  .
\end{align*}

Using a change of variable and equation (\ref{Rdf}), notice that
\[
\psi\left(  \sigma; k,t\right)  =I\left(  \mathcal{K}\left(  \sigma
;t,k\right)  \right)  -\int_{0}^{\infty}h_{0}(s) \ln\left(  \sigma
(\mathcal{K}\left(  \sigma;s, \mathcal{K}\left(  \sigma; t,k \right)  \right)
)\right)  ds.
\]
Therefore, if (\ref{eq10}) holds, then $\psi$ is identically zero, and so is
its derivative $\frac{\partial\psi}{\partial t}$, so (\ref{eq11})\ holds.
Conversely, if (\ref{eq11}) holds, then $\frac{\partial\psi}{\partial t}$
vanishes, and $\psi\left(  k,t\right)  =\psi\left(  k\right)  $ does not
depend on $t$, so that
\begin{align*}
\psi\left(  k\right)   &  =\lim_{t\longrightarrow\infty}\left\{  I\left(
\mathcal{K}\left(  \sigma;t,k\right)  \right)  -\int_{0}^{\infty}h_{0}(s)
\ln\left(  \sigma(\mathcal{K}\left(  \sigma;s, \mathcal{K}\left(  \sigma; t,k
\right)  \right)  )\right)  ds\right\} \\
&  =I\left(  \bar{k}\right)  -\int_{0}^{\infty}h_{0}\left(  s\right)
\ln\left(  \sigma(\bar{k})\right)  dt.
\end{align*}
If in addition, (\ref{eq11BC}) holds then (\ref{eq10}) holds.
\end{proof}

Now let us turn to the the proof of Proposition \ref{prop1}. To simplify the
notations, we shall write $h\left(  t\right)  $ in the following way
\[
h\left(  t\right)  =\theta e^{-(\delta+\pi)t}+\left(  1-\theta\right)
e^{-\rho t}%
\]
where
\[
\theta=\frac{\delta-\rho}{\pi+\delta-\rho},\ 1-\theta=\frac{\pi}{\pi
+\delta-\rho}%
\]

Suppose that $\sigma$ is an equilibrium strategy converging to $\bar k$.
Proposition \ref{Proposition: IE / DE} shows that the value $v$ defined by
(\ref{eq1}) satisfies (\ref{inf2}), (\ref{BCChar}). After performing the optimization in $c$ and
substituting $u(c)=\ln(c)$, the right hand side of (\ref{inf2}) becomes the left hand
side of (\ref{eq:ODE1}). Using the definitions of $v$ and $w$ given by
(\ref{eq1}) and (\ref{eq2}), it can be checked that the left hand side of (\ref{inf2})
coincides with the linear combination of $v$ and $w$ given in the right hand
side of (\ref{eq:ODE1}). Therefore, we proved equation (\ref{eq:ODE1}). The
boundary condition (\ref{BC1}) is proved by integrating (\ref{eq1}) after
replacing the consumption flow $\sigma\left(  \mathcal{K}(\sigma;t,k) \right)
$ by the steady state consumption level $f \left(  \bar k \right)  $.



To get equation (\ref{eq:ODE2}), we use the preceding Lemma with
$\sigma\left(  k\right)  =i\left(  v^{\prime}\left(  k\right)  \right)
=1/v^{\prime}\left(  k\right)  $, $h_{0}(t)=(1-\theta)\left(  e^{-\rho
t}-e^{-(\delta+\pi)t}\right)  $ and we get an equation for $I(k)=w(k)$ given
by%
\begin{equation}
w^{\prime}\left(  k\right)  \left(  f\left(  k\right)  -\frac{1}{v^{\prime
}\left(  k\right)  }\right)  =-\left(  1-\theta\right)  \int_{0}^{\infty
}\left(  \left(  \delta+\pi\right)  e^{-(\delta+\pi)t}-\rho e^{-\rho
t}\right)  u\left(  \sigma\left(  \mathcal{K}\left(  \sigma;t,k\right)
\right)  \right)  dt \label{eq12}%
\end{equation}
and a boundary condition
\[
w\left(  \bar{k}\right)  =\left(  1-\theta\right)  \left(  \frac{1}{\rho
}-\frac{1}{\delta+\pi}\right)  \ln f\left(  \bar{k}\right)  =\frac{\pi}%
{\rho\left(  \delta+\pi\right)  }\ln f\left(  \bar{k}\right)  .
\]
The right hand side of (\ref{eq12}) can be written as the linear combination
$-\pi v(k)+(\rho+\pi)w(k)$ and thus, we proved equation (\ref{eq:ODE2}).


Conversely, suppose $v_{1}$ and $w_{1}$ satisfy the equations(\ref{eq:ODE1})
and (\ref{eq:ODE2}), together with the boundary conditions (\ref{BC1}) and
(\ref{BC2}), with the strategy $\sigma_{1}=1/v_{1}^{\prime}$ converging to
$\bar{k}$, so that%
\begin{equation}%
\begin{array}
[c]{c}%
\left(  f-\frac{1}{v_{1}^{\prime}}\right)  v_{1}^{\prime}-\ln\left(
v_{1}^{\prime}\right)  =\delta v_{1} - (\delta- \rho) w_{1}\\
v_{1}\left(  \bar{k}\right)  =\left(  \frac{\theta}{\delta+\pi}+\frac
{1-\theta}{\rho}\right)  \ln f\left(  \bar{k}\right) \\
\left(  f-\frac{1}{v_{1}^{\prime}}\right)  w_{1}^{\prime}= - \pi v_{1} +
(\rho+ \pi) w_{1}\\
w_{1}\left(  \bar{k}\right)  =\left(  1-\theta\right)  \left(  \frac{1}{\rho}
- \frac{1}{\delta+\pi}\right)  \ln f\left(  \bar{k}\right)
\end{array}
\label{eq20}%
\end{equation}

Consider the functions:%
\begin{align}
v_{2}\left(  k\right)   &  =\int_{0}^{\infty}\left(  \theta e^{-(\delta+\pi
)t}+\left(  1-\theta\right)  e^{-\rho t}\right)  \ln\left(  \sigma_{1}\left(
\mathcal{K}\left(  \sigma_{1};t,k\right)  \right)  \right)  dt\label{eq22}\\
w_{2}\left(  k\right)   &  =\left(  1-\theta\right)  \int_{0}^{\infty}\left(
-e^{-(\delta+\pi)t}+e^{-\rho t}\right)  \ln\left(  \sigma_{1}\left(
\mathcal{K}\left(  \sigma_{1};t,k\right)  \right)  \right)  dt \label{eq23}%
\end{align}
Applying the preceding\ Lemma with $I=v_{2}$ and $I=w_{2}$ successively, we have%

\[%
\begin{array}
[c]{c}%
v_{2}^{\prime}\left(  k\right)  \left(  f-\sigma_{1}\right)  +\int_{0}%
^{\infty}\left(  \theta(\delta+\pi)e^{-(\delta+\pi)t}+\left(  1-\theta\right)
\rho e^{-\rho t}\right)  \ln\left(  \sigma_{1}\left(  \mathcal{K}\left(
\sigma_{1};t,k\right)  \right)  \right)  dt+\ln\left(  \sigma_{1}\left(
k\right)  \right)  =0\\
v_{2}\left(  \bar{k}\right)  =\left(  \frac{\theta}{\delta+\pi}+\frac
{1-\theta}{\rho}\right)  \ln f\left(  \bar{k}\right) \\
w_{2}^{\prime}\left(  k\right)  \left(  f-\sigma_{1}\right)  +\left(
1-\theta\right)  \int_{0}^{\infty}\left(  -e^{-(\delta+\pi)t}+e^{-\rho
t}\right)  \ln\left(  \sigma_{1}\left(  \mathcal{K}\left(  \sigma
_{1};t,k\right)  \right)  \right)  dt=0\\
w_{2}\left(  \bar{k}\right)  =\left(  1-\theta\right)  \left(  \frac{1}{\rho}
- \frac{1}{\delta+\pi} \right)  \ln f\left(  \bar{k}\right)
\end{array}
\]
and hence
\begin{equation}%
\begin{array}
[c]{c}%
v_{2}^{\prime}\left(  k\right)  \left(  f-\sigma_{1}\right)  +\ln\left(
\sigma_{1}\left(  k\right)  \right)  =\delta v_{2}- (\delta- \rho) w_{2}\\
w_{2}^{\prime}\left(  f-\sigma_{1}\right)  = - \pi v_{2}+ (\rho+ \pi) w_{2}%
\end{array}
\label{eq21}%
\end{equation}

Substracting (\ref{eq21}) from (\ref{eq20}), and setting $v=v_{1}%
-v_{2},\ w=w_{1}-w_{2}$, we get:%
\begin{equation}%
\begin{array}
[c]{c}%
\left(  f-\frac{1}{v_{1}^{\prime}}\right)  v^{\prime}=\delta v_{2}- (\delta-
\rho) w_{2}\\
v\left(  \bar{k}\right)  =0\\
\left(  f-\frac{1}{v_{1}^{\prime}}\right)  w^{\prime}=- \pi v_{2}+ (\rho+ \pi)
w_{2}\\
w\left(  \bar{k}\right)  =0
\end{array}
\label{eq25}%
\end{equation}

Obviously $v=w=0$ is a solution. In the next Lemma, we show that it is the
only one, so $v_{1}=v_{2}$ and $w_{1}=w_{2}$. Equation (\ref{eq22}) then
becomes:%
\[
v_{1}\left(  k\right)  =\int_{0}^{\infty}\left(  \theta e^{-(\delta+\pi
)t}+\left(  1-\theta\right)  e^{-\rho t}\right)  \ln\left(  \sigma_{1}\left(
\mathcal{K}\left(  \sigma_{1};t,k\right)  \right)  \right)  dt
\]
which is precisely equation (IE) for $h\left(  t\right)  = \left(  \theta
e^{-(\delta+\pi)t}+\left(  1-\theta\right)  e^{-\rho t}\right)  $ and
$u\left(  c\right)  =\ln c$. Since $v_{1}$ satisfies (IE), the strategy
$\sigma_{1}$ is then an equilibrium strategy.

\begin{lemma}
If $\left(  v,w\right)  $ is a pair of functions, which are continuous on a
neighbourhood $\Omega$ of $\bar{k}$, continously differentiable for $k\neq
\bar{k}$, and which solve (\ref{eq25}) for $k\neq\bar{k}$, then $v=0$ and
$w=0$
\end{lemma}

\begin{proof}
Set $f\left(  k\right)  -1/v_{1}^{\prime}\left(  k\right)  =\varphi\left(
k\right)  $. Note that $\varphi\left(  k\right)  \rightarrow0$ when
$k\rightarrow\bar{k}$. Since $v_{1}$ is $C^{1}$, in fact $C^{2}$, and
$v_{1}^{\prime}\left(  \bar{k}\right)  \neq0$, the value $\varphi\left(
k\right)  $ changes signs when $k$ crosses $\bar{k}$. The system can be
rewritten as:%
\[
\left(
\begin{array}
[c]{c}%
\varphi v^{\prime}\\
\varphi w^{\prime}%
\end{array}
\right)  =\left(
\begin{array}
[c]{cc}%
\delta & -(\delta-\rho)\\
-\pi & \rho+\pi
\end{array}
\right)  \left(
\begin{array}
[c]{c}%
v\\
w
\end{array}
\right)
\]

The characteristic equation of the matrix on the right-hand side is:%
\[
\lambda^{2}-\left(  \rho+\delta+\pi\right)  \lambda+\rho\left(  \delta
+\pi\right)  =0
\]
and the roots are $\lambda=\rho$ and $\lambda=\delta+\pi$. Changing basis, we
can rewrite the system as:%
\[
\left(
\begin{array}
[c]{c}%
\varphi V^{\prime}\\
\varphi W^{\prime}%
\end{array}
\right)  =\left(
\begin{array}
[c]{cc}%
\rho & 0\\
0 & \delta+\pi
\end{array}
\right)  \left(
\begin{array}
[c]{c}%
V\\
W
\end{array}
\right)
\]
where $V$ and $W$ are suitable linear combinations of $v$ and $w$. The
solutions are
\[
V\left(  k\right)  =C_{1}\exp\left(  \rho\int_{k_{0}}^{k}\frac{1}%
{\varphi\left(  l\right)  }dl\right)  ,\ \ W\left(  k\right)  =C_{2} \exp \left(
\left(  \delta+\pi\right)  \int_{k_{0}}^{k}\frac{1}{\varphi\left(  l\right)
}dl\right)
\]
where $C_{1}$, $C_{2}$ and $k_{0}$ are constants. Since $1/\varphi\left(
k\right)  \rightarrow\pm\infty$ when $k\rightarrow\bar{k}$, and both signs
occur, one for $k<\bar{k}$ and the other for $k>\bar{k}$, the only way we can
get $V\left(  \bar{k}\right)  =W\left(  \bar{k}\right)  =0$ is by setting
$C_{1}=C_{2}=0$.
\end{proof}

\section{Proof of Theorem \ref{Th: Existence and multiplicity}}

By Proposition \ref{prop1}, it is enough to show that the system
(\ref{eq:ODE1})-(\ref{eq:ODE2}) with boundary conditions (\ref{BC1}%
)-(\ref{BC2}), has a solution $\left(  v,w\right)  $ where $v$ and $w$ are
required to be $C^{2}$ and such that the resulting capital stock
\[
\frac{dk}{dt}=f\left(  k\right)  -\frac{1}{v^{\prime}\left(  k\right)
},~~k(0)=k_{0}
\]
has the property that $\lim_{t \rightarrow\infty} k\left(  t\right)  = \bar
{k}$ when $k_{0}$ is sufficiently close to $\bar k$.

We will give the proof in several steps. First, we make a sequence of change
of variables and change of coordinates that simplify the system (\ref{eq:ODE1}%
)-(\ref{eq:ODE2}). Second, we apply the Central Manifold Theorem (e.g.
\cite{Carr}) to show existence. Third, we conclude the proof by demonstrating
the estimate of the steady states ${\bar k}$ given in Theorem
\ref{Th: Existence and multiplicity}.

\subsection{Change of variables}

Let us begin with the following useful lemma which proof is omitted.

\begin{lemma}
\label{Lemma: variable X} The equation $x-\ln\left(  1+x\right)
=\mu$ has no solution for $\mu<0$. For $\mu=0$, the solution is
$x=0$. For $\mu>0$, there are two solutions $X_{1}\left(  \mu\right)
$ and $X_{2}\left(  \mu\right)  $ with the following properties:

\begin{description}
\item[(a)] $X_{1}$ is decreasing, $X_{1}\left(  0\right)  =0$, and
$X_{1}\left(  \mu\right)  \longrightarrow-1$ when $\mu\longmapsto\infty$

\item[(b)] $X_{2}$ is increasing, $X_{2}\left(  0\right)  =0$, and
$X_{2}\left(  \mu\right)  \longrightarrow\infty$ when $\mu\longmapsto\infty$

\item[(c)] both $X_{1}$ and $X_{2}$ are continuous on $\left[  0,\ \infty
\right]  $ and $C^{\infty}$ on $\left(  0,\ \infty\right)  $
\end{description}
\end{lemma}

Note, that neither $X_{1}$ nor $X_{2}$ are differentiable at $0$. In fact, if
we replace the function $x-\ln\left(  1+x\right)  =\mu$ by its Taylor
expansion near $x=0$, we find that the equation $x-\ln\left(  1+x\right)
=\mu$ is replaced by the equation $x^{2}=\mu$, so that $X_{1}\left(
\mu\right)  $ \ and $X_{2}\left(  \mu\right)  $ are approximated by
$-\sqrt{\mu}$ and $\sqrt{\mu}$ respectively when $\mu>0$ is small.

Introducing now the function $\mu$ defined by
\begin{align}
\label{eq: mudef}\mu\left(  k\right)  :=\delta v\left(  k\right)  - (\delta-
\rho) w\left(  k\right)  -\ln f\left(  k\right)
\end{align}
equation (\ref{eq:ODE1}) becomes%
\[
fv^{\prime}-1-\ln fv^{\prime}=\mu\label{a10}%
\]
and, using the notation of Lemma \ref{Lemma: variable X}, this equation
becomes
\begin{equation}
f\left(  k\right)  v^{\prime}\left(  k\right)  =1+X_{i}\left(  \mu\left(
k\right)  \right)  ,\ i=1,2 \label{a11}%
\end{equation}

Next, differentiating equation (\ref{eq: mudef}) gives
\[
f\mu^{\prime} =f\left(  \delta v^{\prime}\left(  k\right)  - (\delta-
\rho)w^{\prime}-\frac{f^{\prime}}{f}\right)
\]
and, using equation (\ref{eq:ODE2}) to eliminate $w^{\prime}$ from the above
equation yields
\begin{align}
f\mu^{\prime}  &  =\delta(1+X\left(  \mu\right)  )-f^{\prime}+\frac{1+X\left(
\mu\right)  }{X\left(  \mu\right)  }\left[  -\rho\left(  \pi+ \delta\right)
v+\left(  \pi+ \rho\right)  \mu+\left(  \pi+ \rho\right)  \ln f\right]
\label{a12}%
\end{align}
after replacing $fv^{\prime}$ by $1+X\left(  \mu\right)  $.

Our new system is now (\ref{a11}), (\ref{a12}) where $k$ is the independent
variable, the function $f\left(  k\right)  $ is given, the functions
$X_{1}\left(  \mu\right)  $ and $X_{2}\left(  \mu\right)  $ are defined in
Lemma \ref{Lemma: variable X}, and the unknown functions are $v\left(
k\right)  $ and $\mu\left(  k\right)  $. The right-hand sides of the equations
(\ref{a11}), (\ref{a12}) are function of $\left(  k,v,\mu\right)  $, which are
defined and $C^{\infty}$ for $k>0$ and $\mu>0$.

Suppose now we have a solution $\left(  v,\mu\right)  $, with
\begin{equation}
\mu\left(  \bar{k}\right)  =0. \label{eq: muboundary}%
\end{equation}

Then $X\left(  \mu\left(  \bar{k}\right)  \right)  =0$, so that
\[
v^{\prime}\left(  \bar{k}\right)  =1/f\left(  \bar{k}\right)
\]
by equation (\ref{a11}), and the right-hand side of equation (\ref{a12}%
)\ blows up unless
\[
v\left(  \bar{k}\right)  =\frac{\rho+\pi}{\rho\left(  \delta+\pi\right)  }\ln
f\left(  \bar{k}\right)
\]
which is precisely our boundary condition (\ref{BC1}). If both (\ref{BC1}) and
(\ref{eq: muboundary}) hold, the second boundary condition (\ref{BC2}) is
automatically implied by (\ref{eq: mudef}).

Let us rewrite our new system (omitting the index $i=1.2$)%
\begin{align*}
&  f\left(  k\right)  \frac{dv}{dk}=1+X\left(  \mu\right) \\
&  f\left(  k\right)  \frac{d\mu}{dk}=\delta(1+X\left(  \mu\right)
)-f^{\prime}+\frac{1+X\left(  \mu\right)  }{X\left(  \mu\right)  }\left[
-\rho\left(  \pi+ \delta\right)  v+\left(  \pi+ \rho\right)  \mu+\left(  \pi+
\rho\right)  \ln f\right]
\end{align*}
with the boundary conditions (\ref{BC1}) and (\ref{eq: muboundary}).

We now turn to a change of coordinate by taking $X\left(  \mu\left(  k\right)
\right)  =x$ as the independent variable instead of $k$ along the trajectory.
In other words, instead of looking for functions $v\left(  k\right)  $ and
$\mu\left(  k\right)  $ satisfying the equations (\ref{a11}), (\ref{a12}), we
will be looking for functions $\tilde{v}\left(  x\right)  $ and $\tilde
{k}\left(  x\right)  $ satisfying
\begin{align*}
\frac{d\tilde{v}}{dx}  &  =\frac{dv}{dk}\frac{d\tilde{k}}{dx}=\frac{1+x}%
{f}\frac{d\tilde{k}}{dx},\\
\frac{d\tilde{k}}{dx}  &  =\frac{dk}{d\mu}\frac{d\mu}{dx}=f \frac{d\mu}%
{dx}\left\{  \delta(1+x)-f^{\prime}+\frac{1+x }{x }\left[  -\rho\left(  \pi+
\delta\right)  v+\left(  \pi+ \rho\right)  \mu+\left(  \pi+ \rho\right)  \ln
f\right]  \right\}  ^{-1}\\
&  =f \frac{d\mu}{dx}\frac{x}{1+x}\left\{  \delta x-f^{\prime}\frac{x}{1+x}
-\rho\left(  \pi+ \delta\right)  v+\left(  \pi+ \rho\right)  \mu+\left(  \pi+
\rho\right)  \ln f \right\}  ^{-1}.
\end{align*}

Plugging in the relations $\mu=x-\ln\left(  1+x\right)  $, so that $\frac
{d\mu}{dx} =x\left(  1+x\right)  ^{-1}$, yields the new system%

\begin{align}
\frac{d\tilde{k}}{dx}  &  =f\left(  \tilde{k}\right)  \frac{x^{2}}{1+x}%
\frac{1}{D\left(  x,\tilde{k},\tilde{v}\right)  }\label{b2}\\
\frac{d\tilde{v}}{dx}  &  =x^{2}\frac{1}{D\left(  x,\tilde{k},\tilde
{v}\right)  } \label{b3}%
\end{align}
where the function $D$ is given by
\begin{align}
D\left(  x,k,v\right)  =  &  \left(  1+x\right)  \left[  -\rho\left(  \pi+
\delta\right)  v+\left(  \pi+ \delta+ \rho\right)  x \right] \nonumber\\
&  (1+x) \left(  \pi+ \delta\right)  \left[  \ln f\left(  k\right)
-\ln\left(  1+x\right)  \right]  -xf^{\prime}\left(  k\right)  . \label{b4}%
\end{align}

We pick a point $\bar{k}$ and we look for $C^{2}~$solutions $\left(  \tilde
{k}\left(  x\right)  ,\tilde{v}\left(  x\right)  \right)  $of the system
(\ref{b2}), (\ref{b3}), defined in a neigbourhood of $x=0\,$\ and satisfying:
\begin{align}
\tilde{k}\left(  0\right)   &  =\bar{k}\label{a18}\\
\tilde{v}\left(  0\right)   &  =\frac{\rho+\pi}{\rho\left(  \delta+\pi\right)
}\ln f\left(  \bar{k}\right)  :=\bar{v} \label{a29}%
\end{align}

We now introduce a new variable $s$, and we replace the system (\ref{b2}),
(\ref{b3}) with the following autonomous system
\begin{align}
\frac{dx}{ds}  &  =D\left(  x,k,v\right)  ,\ \ \ x\left(  0\right)
=0\label{a40}\\
\frac{dk}{ds}  &  =\frac{d\tilde{k}}{dx}\frac{dx}{ds}=f\left(  k\right)
\frac{x^{2}}{1+x},\ \ \ k\left(  0\right)  =\bar{k}\label{a41}\\
\frac{dv}{ds}  &  =\frac{d\tilde{v}}{dx}\frac{dx}{ds}=x^{2},\ \ \ v\left(
0\right)  =\bar{v} \label{a42}%
\end{align}
where there are now three unknown functions $\left(  x\left(  s\right)
,k\left(  s\right)  ,v\left(  s\right)  \right)  $, defined near $s=0$. Note
that $D$ is as smooth as $f^{\prime}$ in a neighbourhood of $\left(  0,\bar
{k},\bar{v}\right)  $ and that $D\left(  0,\bar{k},\bar{v}\right)  =0$.

Our problem is then to analyze the system (\ref{a40}),(\ref{a41}) and
(\ref{a42}) since it is related to to the initial system (\ref{eq:ODE1}%
),(\ref{eq:ODE2}) through a sequence of smooth change of variables and coordinates.

\subsection{Existence}

The linearized system near $\left(  0,\bar{k},\bar{v}\right)  $ is:%
\begin{equation}
\frac{d}{ds}\left(
\begin{array}
[c]{c}%
x\\
k\\
v
\end{array}
\right)  =A\left(
\begin{array}
[c]{c}%
x\\
k\\
v
\end{array}
\right)  \text{ } \label{a19}%
\end{equation}
with (all derivatives to be computed at $\left(  0,\bar{k},\bar{v}\right)
$):
\[
A:=\left(
\begin{array}
[c]{ccc}%
\frac{\partial D}{\partial x} & \frac{\partial D}{\partial k} & \frac{\partial
D}{\partial v}\\
0 & 0 & 0\\
0 & 0 & 0
\end{array}
\right)  =\left(
\begin{array}
[c]{ccc}%
\delta-f^{\prime}\left(  \bar{k}\right)  & (\pi+ \rho) \frac{f^{\prime}\left(
\bar{k}\right)  }{f\left(  \bar{k}\right)  } & - \rho(\pi+ \delta)\\
0 & 0 & 0\\
0 & 0 & 0
\end{array}
\right)
\]

The matrix $A$ has the eigenvalues $\left(  \delta-f^{\prime}\left(  \bar
{k}\right)  ,0,0\right)  $, and can obviously be put in diagonal form (in this
case, as a matrix with $1$ in the upper left corner, all the other
coefficients being $0$). To make the change of variables explicit, we note
that the eigenvector associated with the eigenvalue $\delta-f^{\prime}\left(
\bar{k}\right)  $ is%

\[
e=\left(
\begin{array}
[c]{c}%
1\\
0\\
0
\end{array}
\right)
\]
and we set \thinspace$E_{1}:=\mathrm{Span}\left\{  e\right\}  $. We also
consider the kernel of $A$, and denote it by $E_{0}$:%
\[
E_{0}:=\mathrm{Ker}A=\left\{  \left(
\begin{array}
[c]{c}%
x\\
k\\
v
\end{array}
\right)  \ |\ \left(  \delta-f^{\prime}\left(  \bar{k}\right)  \right)
x+(\pi+\rho)\frac{f^{\prime}\left(  \bar{k}\right)  }{f\left(  \bar{k}\right)
}k-\rho(\pi+\delta)v\right\}
\]

Both $E_{1}$ and $E_{0}$ are invariant subspaces of the linearized system
(\ref{a19}) with the corresponding eigenvalues being $\lambda_{1}%
:=\delta-f^{\prime}\left(  \bar{k}\right)  \neq0\,\ $\ and $0$, and
the operator $A$ is diagonal in any base $\left(
e_{1},e_{2},e_{3}\right)  $ with $e_{1}\in E_{1}$ and $\left(
e_{2},e_{3}\right)  \in E_{3}$. By the central manifold theorem (see
for instance \cite{Carr}, Theorem 1)\footnote{Theorem 1 in Carr
\cite{Carr} requires that the linearized matrix at the fixed point
be block diagonal. The matrix $A$ is not block diagonal but the
theorem can still be applied because the matrix $A$  can be
transformed into a block diagonal matrix with some appropriate
change of basis.}, there exists a map $h\left( k,v\right) $, defined
in a neighborhood $\mathcal{O}$ of $\left( \bar{k},\bar{v}\right)  $
such that
\[
h\left(  \bar{k},\bar{v}\right)  =0,\ \ \frac{\partial h}{\partial k}\left(
\bar{k},\bar{v}\right)  =0,\ \ \frac{\partial h}{\partial v}\left(  \bar
{k},\bar{v}\right)  =0
\]
and the manifold $\mathcal{M}$\ defined by:%
\[
\mathcal{M}:=\left\{  \left(  \alpha k+\beta v+h\left(  k,v\right)
,k,v\right)  \ |\ \left(  k,v\right)  \in\mathcal{O}\right\}
\]
is invariant by the flow associated with the equations (\ref{a40}),
(\ref{a41}), (\ref{a42}). The map $h$ and the central manifold $\mathcal{M}$
are as smooth as $f^{\prime}$: they are $C^{2}$, for instance, if $f$ is
$C^{3}$. Note that $\mathcal{M}$ is two-dimensional and tangent to $E_{0}$ at
$\left(  0,\bar{k},\bar{v}\right)  $. Note also that there is another
invariant manifold $N$, which is tangent to $E_{1}$ at $\left(  0,\bar{k}%
,\bar{v}\right)  $:\ it is one-dimensional, and it is stable if $\lambda
_{1}=\delta-f^{\prime}\left(  \bar{k}\right)  <0$ and unstable if $\lambda
_{1}=\delta-f^{\prime}\left(  \bar{k}\right)  >0$. Each of these invariant
manifolds gives a different type of solution to the system equations
(\ref{a40}),\ (\ref{a41}), (\ref{a42}).

We are interested in the solutions which lie on the central manifold
$\mathcal{M}$. They can be found by substituting $x=\alpha k+\beta v+h\left(
k,v\right)  $ in equations (\ref{a41}) and (\ref{a42}), yielding
\begin{align}
\frac{dk}{ds}  &  =f\left(  k\right)  \frac{\left[  \alpha k+\beta v+h\left(
k,v\right)  \right]  ^{2}}{1+\alpha k+\beta v+h\left(  k,v\right)
},\ \ \ k\left(  0\right)  =\bar{k}\label{b5}\\
\frac{dv}{ds}  &  =\left[  \alpha k+\beta v+h\left(  k,v\right)  \right]
^{2},\ \ \ v\left(  0\right)  =\bar{v} \label{b6}%
\end{align}
while $x$ is found by using the fact that $\mathcal{M}$ is invariant
\[
x\left(  s\right)  =\alpha k\left(  s\right)  +\beta v\left(  s\right)
+h\left(  k\left(  s\right)  ,v\left(  s\right)  \right)  .
\]
Eliminating the variable $s$ from (\ref{b5}) and (\ref{b6}), we get%
\[
\frac{dv}{dk}=\frac{f\left(  k\right)  }{1+\alpha k+\beta v+h\left(
k,v\right)  },\ \ v\left(  \bar{k}\right)  =\bar{v}%
\]

The solution of this initial-value problem is $v\left(  k\right)  =\psi\left(
k\right)  $, where $\psi\left(  \bar{k}\right)  =\bar{v}$ and $\psi$ is
$C^{2}$ if $h$ is $C^{2}$, that is, if $f$ is $C^{3}$(see above). Substituting
in $x=h\left(  k,v\right)  $, we get $x\left(  k\right)  =\alpha k+\beta
v+h\left(  k,\psi\left(  k\right)  \right)  $. Finally, $\mu\left(  k\right)
=x\left(  k\right)  -\ln\left(  1+x\left(  k\right)  \right)  $ is also
$C^{\infty}$, so we have found a smooth solution of equations (\ref{a11}) and
(\ref{a12}), as desired.

Differentiating equations (\ref{eq:ODE1}) and evaluating it at $k=\bar{k}$,
taking into account that $f\left(  \bar{k}\right)  v^{\prime}\left(  \bar
{k}\right)  =1$, yields the following%
\[
w^{\prime}\left(  \bar{k}\right)  =\frac{1}{(\delta- \rho)f\left(  \bar
{k}\right)  }\left(  \delta- f^{\prime}\left(  \bar{k}\right)  \right)
\]

\subsection{Proving the estimate}

It remains to prove that the strategy $\sigma\left(  k\right)  =1/v^{\prime
}\left(  k\right)  $ is convergent. Recall that this means that the solutions
of the equation
\[
\frac{dk}{dt}=f\left(  k\right)  -1/v^{\prime}\left(  k\right)  ,~~k(0)=k_{0}
\]
converge to $\bar{k}$ if $k_{0}$ i sufficiently close to $\bar k$. Since
$f\left(  \bar{k}\right)  =1/v^{\prime}\left(  \bar{k}\right)  $, $\bar{k}$ is
a fixed point of the dynamical system, and we want to show that it is an
attractor. This means that the linearized system at $\bar{k}$, namely
\[
\frac{dk}{dt}=\left(  f^{\prime}\left(  \bar{k}\right)  +\frac{v^{\prime
\prime}\left(  \bar{k}\right)  }{v^{\prime}\left(  \bar{k}\right)  ^{2}%
}\right)  k
\]
must have $k=\bar k$ as an attractor. In other words, we must have
\begin{equation}
f^{\prime}\left(  \bar{k}\right)  +\frac{v^{\prime\prime}\left(  \bar
{k}\right)  }{v^{\prime}\left(  \bar{k}\right)  ^{2}}<0 \label{e7}%
\end{equation}

To compute the left-hand side of (\ref{e7}), differentiate equation
(\ref{eq:ODE2}) at set $k=\bar{k}$. We get%

\[
\left(  f^{\prime}\left(  \bar{k}\right)  +\frac{v^{\prime\prime}\left(
\bar{k}\right)  }{v^{\prime}\left(  \bar{k}\right)  ^{2}}\right)  w^{\prime
}\left(  \bar{k}\right)  = -\pi v^{\prime}(\bar k) + (\rho+ \pi) w^{\prime
}(\bar k).
\]

To find $w^{\prime}(\bar{k})$, differentiate equation (\ref{eq:ODE1}) and
evaluate t at $k=\bar{k}$, taking into account that $f\left(  \bar{k}\right)
v^{\prime}\left(  \bar{k}\right)  =1$, to get
\[
w^{\prime}\left(  \bar{k}\right)  =\frac{1}{(\delta-\rho)f\left(  \bar
{k}\right)  }\left(  \delta-f^{\prime}\left(  \bar{k}\right)  \right)  .
\]
Hence%
\[
\left(  f^{\prime}\left(  \bar{k}\right)  +\frac{v^{\prime\prime}\left(
\bar{k}\right)  }{v^{\prime}\left(  \bar{k}\right)  ^{2}}\right)  w^{\prime
}\left(  \bar{k}\right)  =\frac{\rho(\delta+\pi)-(\rho+\pi)f^{\prime}(\bar
{k})}{f(\bar{k})(\delta-\rho)}.
\]
and, switching $w^{\prime}(\bar{k})$ from the left hand side to the right hand
side of the last equation gives
\begin{equation}
f^{\prime}\left(  \bar{k}\right)  +\frac{v^{\prime\prime}\left(  \bar
{k}\right)  }{v^{\prime}\left(  \bar{k}\right)  ^{2}}=\frac{\rho(\delta
+\pi)-(\rho+\pi)f^{\prime}(\bar{k})}{\delta-f^{\prime}\left(  \bar{k}\right)
} \label{eqr}%
\end{equation}
This will be negative if the numerator and denominator have opposite signs.
Both the numerator and the denominators of the right hand side of the last
equation are increasing functions of $\bar{k}$ and they change sign
respectively at $f^{\prime}(\bar{k})=\rho\frac{\delta+\pi}{\pi+\rho}$ and
$f^{\prime}(\bar{k})=\delta$. Since $\rho<\pi$, we have $\rho\frac{\delta+\pi
}{\pi+\rho}<\delta$ and therefore the only interval where $f^{\prime}\left(
\bar{k}\right)  +\frac{v^{\prime\prime}\left(  \bar{k}\right)  }{v^{\prime
}\left(  \bar{k}\right)  ^{2}}<0$ is the open interval $I=\left(  \rho
\frac{\delta+\pi}{\pi+\rho},\delta\right)  $.

\section{Proof of Theorem \ref{NL}}

\bigskip

We will need a preliminary result, the proof of which is quite obvious.

\begin{lemma}
Let $f\left(  x,y\right)  $ be a $C^{1}$ function of two variables such that:%
\begin{align*}
f\left(  x,x\right)   &  =\varphi\left(  x\right) \\
\frac{\partial f}{\partial x}\left(  x,x\right)   &  =\psi\left(  x\right)
\end{align*}
Then:%
\[
\frac{\partial f}{\partial y}\left(  x,x\right)  =\varphi^{\prime}\left(
x\right)  -\psi\left(  x\right)
\]

\end{lemma}

To prove that $\bar\sigma$ is not l.r.p, we will establish that there exists
$\varepsilon>0$ such the inequality $(\ref{eq: lrp})$ is not satisfied for all
$k_{0} \in(\bar k - \varepsilon,\bar k + \varepsilon)$ and $k^{*} \in(\bar k ,
\bar k + \varepsilon) $. We will now introduce a new notation of the
equilibrium value as a function of the initial point $k_{0}$, as before, and
of the terminal point $\bar{k}$. More precisely, we define
\[
V\left(  k_{0},\bar{k}\right)  =\int_{0}^{\infty}h\left(  t\right)  u\left(
\bar{\sigma}\left(  \mathcal{K}\left(  \bar{\sigma};t,k_{0}\right)  \right)
\right)  dt
\]
where $h\left(  t\right)  $ is given by (\ref{IEOG}) and $\bar{\sigma}$ is an
equilibrium strategy converging to $\bar{k}$ \footnote{The function $V$ is
defined for any $k^{*} \in I$ by
\[
V(k_{0},k^{*}) = \int_{0}^{\infty}h\left(  t\right)  u\left(  \sigma
^{*}\left(  \mathcal{K}\left(  \sigma^{*};t,k_{0}\right)  \right)  \right)
dt
\]
where $\sigma^{*}$ is an equilibrium policy converging to $k^{*}$. Of course,
there may be multiple equilibria converging to $k^{*}$ and in order to define
properly $V( . , .)$ we need to make an a priori selection of a converging
equilibrium policy for each $k^{*} \in I$}. Assuming that $V$ is
differentiable with respect to $\bar k$, we can apply the preceding Lemma,
with
\[
V\left(  \bar{k},\bar{k}\right)  =\frac{\rho+\pi}{\rho\left(  \delta
+\pi\right)  }\ln f\left(  \bar{k}\right)  ,\ \frac{\partial V}{\partial
k_{0}}\left(  \bar{k},\bar{k}\right)  =\frac{1}{f\left(  \bar{k}\right)  }%
\]
and get%
\[
\frac{\partial V}{\partial\bar{k}}\left(  \bar{k},\bar{k}\right)  =\frac
{1}{f\left(  \bar{k}\right)  }\left(  \frac{\pi+ \rho}{\rho\left(  \pi+
\delta\right)  }f^{\prime}\left(  \bar{k}\right)  -1\right)
\]
which is positive on the whole allowable interval $I$. It follows that there
exists $\varepsilon>0$
\[
\frac{\partial V}{\partial\bar{k}}\left(  k_{0},\tilde{k}\right)  > 0
\]
for all $k_{0},\tilde{k}$ in the interval $(\bar k - \varepsilon, \bar k +
\varepsilon)$.

Now, consider any initial capital level $k_{0} \in(\bar k - \varepsilon, \bar
k + \varepsilon) $ and steady state capital level $k^{*} \in(\bar k , \bar k +
\varepsilon)$. By the mean value theorem, we have
\[
V(k_{0},k^{*})-V(k_{0}, \bar k) = \left(  k^{*} - \bar k\right)
\frac{\partial V}{\partial\bar{k}}\left(  k_{0},\tilde{k}\right)  >0
\]
for some $\tilde k \in(\bar k, k^{*})$ and where the inequality follows from
the fact that $\tilde k \in(\bar k - \varepsilon, \bar k + \varepsilon)$.

\bigskip

\bigskip

\appendix{}


\begin{thebibliography}{99}                                                                                               %


\bibitem {Anslie1}Ainslie, G. (1975): \textquotedblleft Specious Reward:A
Behavioral Theory of Impulsiveness and Impulse Control,\textquotedblright%
\ \textit{Psychological Bull,} 82, 463-96.





\bibitem {Asheim1}Asheim G. B. (1997): "Individual and Collective
Time-Consistency,\textquotedblright\ \textit{Review of Economic
Studies, }64, 427-443.

\bibitem {Asheim2}Asheim G. B. (1988): "Rawlsian Intergenerational Justice as
a Markov-Perfect Equilibrium in a Resource Technologys,\textquotedblright%
\ \textit{Review of Economic Studies}, 55, 469-483.

\bibitem {AtkinsonSandmo}Atkinson, A.B. and Sandmo, A. (1980):
\textquotedblleft Welfare Implications of the Taxation of
Savings,\textquotedblright\ \textit{Economic Journal, }90, 529-549.

\bibitem {AuerbachKotlikoff}Auerbach, A. and Kotlikoff, L. (1987): "Dynamic
Fiscal Policy,\textquotedblright\ \textit{New York: Cambridge
University Press}.

\bibitem {Aumann}Aumann, R. (1964): "Markets with a Continuum of Traders",
\textit{Econometrica}, 32, 39-50.

\bibitem {Barro}Barro, R. J. (1999): \textquotedblleft Ramsey Meets Laibson in
the Neoclassical Growth Model,\textquotedblright\ \textit{Quarterly
Journal of Economics,} 114, 1125-52.



\bibitem {Bernheim}Bernheim, B. D. (1989): \textquotedblleft Intergenerational
Altruism, Dynastic Equilibria and Social Welfare,\textquotedblright%
\ \textit{Review of Economic Studies}, 56, 119-128.

\bibitem {Bernheim-Rey}Bernheim, B. D., and Ray, D. (1986): \textquotedblleft
On the Existence of Markov-Consistent Plans under Production
uncertainty,\textquotedblright\ \textit{Review of Economic Studies},
53, 877-882.

\bibitem {Bernheim-Rey2}Bernheim, B. D., and Ray, D. (1989): \textquotedblleft
Collective Dynamic Consistency in Repeated Games,\textquotedblright%
\ \textit{Games and Economic Behavior,} 1, 295--326.

\bibitem {Bernheim-Rey1}Bernheim, B. D., and Ray, D. (1989): \textquotedblleft
Markov Perfect Equilibria in Altruistic Growth Economies with
Production Uncertainty,\textquotedblright\ \textit{Journal of
Economic Theory,} 47, 195-202.



\bibitem {Blanchard}Blachard, O. J. (1985): \textquotedblleft Debt, Deficits,
and Finite Horizons,\textquotedblright\ \textit{Journal of Political
Economy}, 93, 223-247.




\bibitem {Calvo}Calvo, G. (1978): \textquotedblleft
On the Time Consistency o Optimal Policy in a Monetray Economy,\textquotedblright\ \textit{Econometrica},
6, 1411-1428.

\bibitem {CalvoObstfeld}Calvo, G. and Obstfeld, M. (1988): \textquotedblleft
Optimal Time Consistent Fiscal Policy with Finite Lifetime,\textquotedblright\
\textit{Econometrica},
56, 411-432.

\bibitem {Caplin-Leahy1}Caplin, A. and Leahy, J. (2004): \textquotedblleft The
Social Discount Rate,\textquotedblright\ \textit{Journal of
Political Economy}, 112, 1257-1268.

\bibitem {Caplin-Leahy}Caplin, A. and Leahy, J. (2006): \textquotedblleft The
Recursive Approach to Time Inconsistency,\textquotedblright\
\textit{Journal of Economic Theory}, 131, 134-156.

\bibitem {Carr}Carr, J. (1981): \textquotedblleft Applications of Centre
Manifold Theory," \textit{Springer-Verlag}.





\bibitem {Chamley}Chamley, C. (1986): "Optimal Taxation of Capital Income in
General Equilibrium with Infinite Lives,\textquotedblright\textit{
Econometrica}, 54, 607-622.

\bibitem {ChariKehoe}Chari, W. and Kehoe, P.J. (1990): "Sustainable
Plans,\textquotedblright\ \textit{Journal of Political Economy, }98,
783-802.

\bibitem {Chichi}Chichilnisky, G. (1996): \textquotedblleft An Axiomatic
Approach to Sustainable Development,\textquotedblright\
\textit{Social Choice and Welafare}, 13, 219-248.













\bibitem {Dasgupta}Dasgupta, P. (1974): \textquotedblleft On Some Alternative
Criteria for Justice Between Generations,\textquotedblright\
\textit{Journal of Public Economics}, 3, 405-423.

\bibitem {DeMa}DeMarzo, P. and Uro\v{s}evi\'{c}, B. (2006): \textquotedblleft
Ownership Dynamics and Asset Pricing with a Large
Shareholder,\textquotedblright\ \textit{Journal of Political
Economy}, 114, 774-815.

\bibitem {Diamond}Diamond, P. A. (1965): \textquotedblleft National Debt in a
Neoclassical Growth Model,\textquotedblright\ \textit{American
Economic Review}, 55, 1126-1150.

\bibitem {Diamond1}Diamond, P.A. (1973): \textquotedblleft Taxation and Public
Production in a Growth Setting\textquotedblright, \textit{Models of
Economic Growth, }ed. by J.A. Mirrlees and N.H.Stern. London:
MacMillan.\textit{ }

\bibitem {Diamond-Koszegi}Diamond , P.A. and Koszegi B. (2003):
\textquotedblleft Quasi-hyperbolic Discounting and Early
Retirement,\textquotedblright\ \textit{Journal of Public Economics},
9, 1839--72.





\bibitem {ErosaGervais}Erosa, A. and Gervais, M. (2002): "Optimal Taxation in
Life-Cycle Economies," \textit{Journal of Economic Theory}, 105,
338-369.



\bibitem {Farrell}Farell, J. and Maskin, E. (1989): "Renegotiation in Repeated
Games," \textit{Games and Economic Behavior, }1, 327-360.





\bibitem {Fischer}Fischer, S. (1980): "Dynamic Inconsistency, Cooperation and
the Benevolent Dissembling Government", \textit{Journal of Economic
Dynamics and Control}, 2, 93-107.

\bibitem {Goldman}Goldman, S. (1980): \textquotedblleft Consistent
Plans,\textquotedblright\ \textit{The Review of Economic Studies},
47, 533-537.

\bibitem {Gul1}Gul, F. and Pesendorfer, W. (2001): \textquotedblleft
Temptation and Self-Control,\textquotedblright\
\textit{Econometrica,} 69, 1403-35.



\bibitem {Harris}Harris, C. (1985): \textquotedblleft Existence and
Characterization of Perfect Equilibrium in Games with Perfect
Information,\textquotedblright\ \textit{Econometrica,} 53, 613-628.

\bibitem {Harris-Laibson}Harris, C. and Laibson, D. (2001): \textquotedblleft
Dynamic Choices of Hyperbolic Consumers,\textquotedblright%
\ \textit{Econometrica,} 69, 935-57.

\bibitem {Harris-Laibson1}Harris, C. and Laibson, D. (2002): \textquotedblleft
Hyperbolic Discounting and Consumption,\textquotedblright\ eds.
Mathias Dewatripont, Lars Peter Hansen, and Stephen Turnovsky,
\textit{Advances in Economics and Econometrics: Theory and
Applications, Eighth World Congress,} Volume 1, 258-298.

\bibitem {Harris-Laibson2}Harris, C. and Laibson, D. (2004): \textquotedblleft
Instantanoeus Gratification,\textquotedblright\ \textit{Working
paper, Harvard University.}



\bibitem {Judd}Judd, K. (1985): "Redistributive Taxation in a Simple Perfect
Foresight Model,"\textit{ Journal of Public Economics}, 28, 59-83.

\bibitem {Karp}Karp, L. S. (2007): \textquotedblleft Non-constant Discounting
in Continuous Time\textquotedblright, \textit{Journal of Economic
Theory, }132, 557-568.


\bibitem {Kihlstrom}Kihlstrom, R. (2007): \textquotedblleft
Risk Aversion and the Elasticity of Substitution in General Dynamic
Portfolio Theory: Consistent Planning by Forward Looking, Expected
Utility Maximizing Investors ,\textquotedblright\ \textit{Working
paper, University of Pennsylvania.}

\bibitem {Kocherlakota}Kocherlakota, N. (1996): \textquotedblleft
Reconsideration-Proofness: A Refinement for Infinite Horizon Time
Inconsistency,\textquotedblright\ \textit{Games and Econonomic
Behaviour,} 15, 33-54.

\bibitem {Kocherlakota1}Kocherlakota, N. (2008): "Monetary and Fiscal Policy:
An Overview," \textit{2nd Edition of New Palgrave Dictionary of
Economics}.



\bibitem {Krusell-Smith 0}Krusel, P., Kuru\c{s}\c{c}u, B. and Smith, A.
(2002): "Equilibrium Welfare and Government Policy with
Quasi-Geometric Discounting,\textquotedblright\ \textit{Journal of
Economic Theory}, 105, \ 42--72.

\bibitem {Krusell-Smith}Krusell, P. and Smith, A. (2003): \textquotedblleft
Consumption-Savings Decisions with Quasi-geometric
Discounting,\textquotedblright\ \textit{Econometrica,} 71, 365-75.



\bibitem {Kyd}Kydland, F. and Prescott, E. (1977): \textquotedblleft Rules
Rather than Discretion: The Inconsistency of Optimal
Plans\textquotedblright, \textit{Journal of Political Economy}, 85,
473-490.

\bibitem {Laibson0}Laibson, D. (1996): \textquotedblleft Hyperbolic Discount
Functions, Undersaving and Savings Policy,\textquotedblright\
\textit{NBER working paper}, w5635

\bibitem {Laibson1}Laibson, D. (1997): \textquotedblleft Golden Eggs and
Hyperbolic Discounting,\textquotedblright\ \textit{Quarterly Journal
of Economics,} 112, 443-77.

\bibitem {Laibson2}Laibson, D. (1998): \textquotedblleft Life-cycle
Consumption and Hyperbolic Discount Functions,\textquotedblright%
\ \textit{European Economic Review} 42, 861-871.

\bibitem {LaneMitra}Lane, J. and Mitra, T. (1981): "On Nash Equilibrium
Programs of Capital Accumulation," \textit{International Economic
Review}, 22, 309-331.




\bibitem {LiLo}Li, C. Z. and L\"{o}fgren, K.G. (2000): \textquotedblleft
Renewable Resources and Economic Sustainability: A Dynamic Analysis
with Heterogeneous Time Preferences,\textquotedblright\
\textit{Journal of Environmental Economics and Management}, 40,
236-250.





\bibitem {Luttmer-Mariotti}Luttmer, E. G. J., and Mariotti, T. (2003):
\textquotedblleft Subjective Discounting in an Exchange
Economy,\textquotedblright\ \textit{Journal of Political Economy,
}111, 959-89.



\bibitem {O'Donoghue-Rabin1}O'Donoghue, T. and Rabin, M. (1999):
\textquotedblleft Doing It Now or Later,\textquotedblright\
\textit{American Economic Review,} 89, 103-24.



\bibitem {Peleg-Yaari}Peleg, B. and Yaari, M. (1973): \textquotedblleft On the
Existence of a Consistent Course of Actions When Tastes Are
Changing,\textquotedblright\ \textit{Review of Economic Studies} 40,
391-401.

\bibitem {Phelps}Phelps, E. S., (1975): "The Indeterminacy of Game-Equilibrium
Growth in the Absence of an Ethic,\textquotedblright\ \textit{In E.
Phelps (Ed.), Altruism, morality and economic theory}.\textit{ }New
York: Russell Sage Foundation

\bibitem {Phelps-Pollak}Phelps, E. S. and Pollak, R. A. (1968):
\textquotedblleft On Second-Best National Saving and
Game-Equilibrium Growth, \textquotedblright\ \textit{Review of
Economic Studies,} 35, 185-99.



\bibitem {Pollak}Pollak, R. A. (1968): \textquotedblleft Consistent
Planning,\textquotedblright\ \textit{Review of Economic Studies,}
35, 201-208.



\bibitem {Ramsey}Ramsey F.P. (1928): "A Mathematical Theory of
Saving,\textquotedblright\ \textit{Economic Journal, }38, 543-559.



\bibitem {Samuelson}Samuelson, P. A. (1958): \textquotedblleft An Exact
Consumption-Loan Model of Interest with or without the Social
Contrivance of Money,\textquotedblright\ \textit{Journal of
Political Economy}, 66, 467-482.

\bibitem {Simon}Simon, L. K. and Stinchcombe M. B. (1989): \textquotedblleft
Extensive Form Games in Continuous Time: Pure Strategies,\textquotedblright%
\ \textit{Econometrica,} 57, 1171-1214.

\bibitem {Sto}Stokey, N. (1981): \textquotedblleft Rational Expectations and
Durable Goods Pricing\textquotedblright, \textit{Bell Journal of
Economics}, 12, 112-128.

\bibitem {Strotz}Strotz, R. H. (1956): \textquotedblleft Myopia and
Inconsistency in Dynamic Utility Maximization.\textquotedblright%
\ \textit{Review of Economic Studies,} 23, 165-80.

\bibitem {Sumaila}Sumaila, U. and Walters, C. (2005): "Intergenerational
Discounting:\ a New Intuitive Approach,\textquotedblright\
\textit{Ecological Economics,} 52, 135-142.

\bibitem {Yaari}Yaari, M. (1965): "Uncertain Lifetimes, Life Insurance, and the theory of the consumer,
\textquotedblright\ \textit{Review of Economic Studies,} 32,
137-150.



\end{thebibliography}
\end{document}